\documentclass[a4paper,12pt]{article}
\usepackage[utf8]{inputenc}

\usepackage{amssymb}
\usepackage{amsmath}
\usepackage{amsthm}
\usepackage{tikz}
\usepackage{pgfplots}
\usepackage{tikz-cd}
\usepackage{hyperref}
\usepackage[maxnames=50]{biblatex}
\title{Isotriviality of families of curves parametrized by $\Ag(n)$}
\author{\'Eloan Rapion}

\bibliography{biblio}
\AtBeginBibliography{\small}

\newcommand{\dif}{\mathop{}\mathopen{}\mathrm d}

\newcommand*{\application}[5]{\begin{array}{lrcl}
		#1: & #2 & \to & #3 \\
		& #4 & \mapsto & #5
	\end{array}
}

\newcommand{\ext}[1]%
{{\vphantom{#1}}^{\mathit \diamond}{#1}}

\newtheorem{thm}{Theorem}[section]

\newtheorem{cor}[thm]{Corollary}
\newtheorem{defi}[thm]{Definition}

\newtheorem{lem}[thm]{Lemma}
\newtheorem{prop}[thm]{Proposition}

\DeclareMathOperator{\AugBL}{\mathbb{B}_+}
\DeclareMathOperator{\Aut}{Aut}
\DeclareMathOperator{\BL}{Bs}
\DeclareMathOperator{\Charac}{\mathcal{S}}
\DeclareMathOperator{\Chern}{C_1}
\DeclareMathOperator{\chern}{c_1}
\DeclareMathOperator{\End}{End}
\DeclareMathOperator{\Exc}{Exc}
\DeclareMathOperator{\Frac}{Frac}

\DeclareMathOperator{\GL}{GL}
\DeclareMathOperator{\home}{h}
\DeclareMathOperator{\Img}{Im}

\DeclareMathOperator{\PABL}{\mathbb{B}_+^{\Proj}}
\DeclareMathOperator{\PBL}{\StBL^{\Proj}}
\DeclareMathOperator{\PmBL}{\BL_m^{\Proj}}

\DeclareMathOperator{\Proj}{\mathbb{P}}
\DeclareMathOperator{\Spec}{Spec}
\DeclareMathOperator{\StBL}{\mathbb{B}}
\DeclareMathOperator{\Sym}{Sym}
\DeclareMathOperator{\TBL}{\StBL^T}
\DeclareMathOperator{\TmBL}{\BL_m^T}

\renewcommand{\epsilon}{\varepsilon}

\def\Ag{\mathcal{A}_g}
\def\Agn{\Ag(n)}

\def\Amp{A}
\def\Anman{M}

\def\Base{S}
\def\C{\mathbb{C}}
\def\Came{g}
\def\came{\hat{\Came}}
\def\Charac{\mathcal{S}}
\def\chri{\varphi}
\def\Cinf{\mathcal{C}^\infty}
\def\Comp{\bar{\Lsav}}
\def\Const{C}
\def\di{d}
\def\Disk{\Delta}
\def\Disr{\Disk_r}
\def\Div{D}
\def\Dipo{\chi}
\def\Dom{\Sigma}
\def\Dual{\check{\Dom}}

\def\Fib{\Proj \Omega_{\Dom}^1}
\def\Fico{\widetilde{\Lsav}}
\def\Filho{\textbf{F}}
\def\gsec{\sigma}
\def\glie{\mathfrak{g}}
\def\Group{G}

\def\im{\textbf{i}}
\def\incl{\iota}

\def\Isot{K}
\def\Latt{\Gamma}
\def\LB{L}
\def\llie{\mathfrak{l}}
\def\Locsys{\textbf{H}}
\def\Lsav{X}
\def\ma{\mathfrak{m}}

\def\Mg{\mathcal{M}_g}

\def\Mh{\mathcal{M}_h}
\def\mlie{\mathfrak{m}}
\def\N{\mathbb{N}}
\def\Ns{\N_{>0}}
\def\nul{n}
\def\Null{\mathcal{N}}
\def\Omun{\Stru(-1)}
\def\Pdis{\Disr^*}
\def\Pgroup{P}
\def\plie{\mathfrak{p}}
\def\Pol{\textbf{S}}
\def\Ppdi{(\Pdis)^\di}
\def\PrB{\Proj \Omega_{\Comp}^1(\log \Div)}
\def\Prv{W}

\def\QprB{\Proj \Omega_{\Lsav}^1}
\def\Qpv{V}
\def\R{\mathbb{R}}
\def\Ralg{\R_{\text{alg}}}
\def\Ranexp{\R_{\text{an, exp}}}
\def\reg{0}
\def\Ricu{R}
\def\sing{\text{sing}}
\def\Smofo{A}
\def\Stru{\mathcal{O}}
\def\Taut{\Stru(1)}
\def\Thin{T}
\def\U{\mathbb{U}}
\def\vame{\mu}

\def\VB{E}

\def\Z{\mathbb{Z}}

\begin{document}

\maketitle

\begin{abstract}
	We prove that for every integers $g, h\geq 2, n \geq 3$, for all but finitely many prime numbers $p$, for every field $k$ of characteristic $0$ or $p$, every separable family of smooth projective curves of genus $h$ over $\Agn \otimes k$ is isotrivial. To prove this, we compute the common vanishing locus of the absolutely logarithmic symmetric forms on a smooth complex algebraic variety whose universal covering is biholomorphic to an irreducible bounded symmetric domain of rank at least $2$.
\end{abstract}

\section{Introduction}

Consider a field $k$ and a $k$-variety $\Qpv$ (an integral separated scheme $\Qpv$ of finite type over $\Spec(k)$). For $g \geq 2$ an integer, we call \emph{family of smooth projective curves} of genus $g$ over $\Qpv$ the data of a $k$-variety $\Prv$ and a smooth projective morphism $\Prv \to \Qpv$ whose geometric fibers are connected smooth projective curves of genus $g$. Let $\Mg$ be the moduli stack of connected smooth projective curves of genus $g$. The data of an isomorphism class of families of smooth projective curves over $\Qpv$ is equivalent to the data of a $k$-morphism $\Qpv \to \Mg \otimes k$.

\bigskip

A family of smooth projective curves of genus $g$ over $\Qpv$ associated with a morphism $f: \Qpv \to \Mg \otimes k$ is \emph{separable} if for every $k$-variety $\Qpv'$ (resp. $M$), for every dominant étale morphism $p: \Qpv' \to \Qpv$ (resp. $q: M \to \Mg \otimes k$) and every lift $f': \Qpv' \to M$ of $f$, the corestriction $f': \Qpv' \to \overline{\Img f'}$ is separable. This is equivalent to the existence of such $\Qpv, M, p, q$ and $f'$ that realize this condition.

\bigskip

Let $M_g$ be the coarse moduli space associated with $\Mg$. A family of smooth projective curves of genus $g$ over $\Qpv$ is \emph{isotrivial} if the composition of the morphism $\Qpv \to \Mg \otimes k$ associated with the family of smooth projective curves and of the natural morphism $\Mg \otimes k \to M_g \otimes k$ is constant. We denote by $\Ag$ the Deligne-Mumford stack of principally polarized Abelian varieties of dimension $g$. For $n \in \Ns$, we denote by $\Agn$ the moduli stack of principally polarized Abelian varieties with $n$-level structure, which is a scheme for $n \geq 3$. Our goal is to prove the following theorem.

\begin{thm}\label{thm:mainapp}
	For every integers $g, h\geq 2, n \geq 3$, for all but finitely many prime numbers $p$, for every field $k$ of characteristic $0$ or $p$, every separable family of smooth projective curves of genus $h$ over $\Agn \otimes k$ is isotrivial.
\end{thm}

To prove the previous theorem, we consider the logarithmic symmetric forms on logarithmic compactifications of $\Ag(n) \otimes k$ and of a finite étale covering $M$ of $\Mg \otimes k$. Indeed, for every non zero tangent vector $v$ of $M$, there is a logarithmic symmetric form which does not vanish on $v$. The data of a family of smooth projective curves on $\Ag(n) \otimes k$ allows to pull back logarithmic symmetric forms of $M$ on a finite étale covering of $\Ag(n) \otimes k$ (Proposition \ref{prop:functform}). Some power of this pull back descends on $\Ag(n) \otimes k$ (Proposition \ref{prop:covform}). But we will prove that logarithmic symmetric forms must vanish on a subset of $\Proj \Omega_{\Ag(n) \otimes k}^1$ which linearly generates it above each point of $\Ag(n) \otimes k$ (although in characteristic $0$, there are ``a lot'' of logarithmic symmetric forms in the sense that $\Omega_{\Ag(n) \otimes k}^1$ is L-big). This vanishing constraint implies that the families of smooth projective curves are isotrivial (Proposition \ref{prop:cste}).

\bigskip

We can deduce the vanishing constraint of logarithmic symmetric forms of $\Ag(n) \otimes k$ from the case $k = \C$ with Propositions \ref{prop:ext} and \ref{prop:spe} (but Theorem \ref{thm:mainapp} itself is not a consequence of the same statement in the case $k = \C$). The universal covering of $\Ag(n) \otimes \C$ is biholomorphic to $g$-th Siegel upper half-space. Mok determined the stable and augmented base loci of the tautological line bundle over the projectivization of the cotangent bundle of every proper smooth complex algebraic variety $\Lsav$ whose universal covering is biholomorphic to an irreducible bounded symmetric domain of rank at least $2$ (\cite[Proposition 4, p.262]{mok1989metric}): it is the characteristic subvariety $\Charac$, which can be defined in the following way. Consider $\Ricu$ the Riemann curvature tensor on $\Lsav$ for the metric induced by Bergman metric on its universal covering. Hence $\Charac \subset \QprB$ is the set of vector lines $[v]$ generated by holomorphic tangent vectors $v \in T_{\Lsav}$ such that the rank of the Hermitian form $(u, w) \mapsto \Ricu(v, \bar{v}, u, \bar{w})$ is not maximal. We cannot directly use the theorem of Mok, as $\Agn \otimes \C$ is not proper. We prove the following generalization of Mok's result without the assumption of properness. 

\begin{thm}\label{thm:main}
	Let $\Lsav$ be a smooth complex algebraic variety whose universal covering is biholomorphic to an irreducible bounded symmetric domain of rank at least $2$. Consider a projective logarithmic compactification $(\Comp, \Div)$ of $\Lsav$. Consider $\Taut$ the tautological line bundle over $\Proj\Omega_{\Comp}^1(\log \Div)$. Consider the projection $\pi: \Proj\Omega_{\Comp}^1(\log \Div) \to \Comp$. Hence the intersection of both the stable base locus and the augmented base locus of $\Taut$ with $\pi^{-1}(\Lsav)$ is the characteristic subvariety $\Charac$.
\end{thm}

Mok's proof of the compact case is based on a Finsler metric rigidity theorem. The obstruction to extend his proof in a non compact context is the use of integrals over closed subvarieties: the integrability of differential forms is obvious in the compact case but not in the non compact case. To answer this problem, we use real polarized variations of Hodge structures (PVHS). Indeed, in \cite[\S 4]{zucker1981locally}, Zucker allows us to see the logarithmic tangent bundle and its metric as a holomorphic subquotient bundle of the holomorphic bundle assoicated with the local system of some PVHS, with the metric induced by the Hodge metric. Then, we can use the careful study of the singularities of this metric done by Cattani, Kaplan, Schmid \cite{cattani1986degeneration} and Koll{\'a}r \cite{kollar1987subadditivity}.

\bigskip

Using the same proof, Theorem \ref{thm:mainapp} generalizes to every integral model of Shimura variety which admits a convenient compactification, instead of $\Agn$ (see for instance \cite{lan2008arithmetic}). In the case $k = \C$, Theorem \ref{thm:main} also implies the two following theorems. The first one is a generalization of Theorem \ref{thm:mainapp} and of \cite[Corollary 1.6]{liu2017curvatures}. The second one is a consequence of \cite[Theorem A]{brotbek2018positivity}.

\begin{thm}\label{thm:gen}
	A family of smooth projective curves of genus $g \geq 2$ over a smooth complex algebraic variety $\Lsav$ whose universal covering is an irreducible bounded symmetric domain of rank at least $2$ is isotrivial.
\end{thm}

\begin{thm}\label{thm:bro}
	Let $\Lsav$ be a smooth complex algebraic variety whose universal covering is an irreducible bounded symmetric domain of rank at least $2$. Let $\Prv$ be a smooth projective variety of dimension $\di$. Let $L$ be an ample divisor on $\Prv$. Let $c \geq \di$, $m \geq (4\di)^{\di+2}$. Then, for a general divisor $\Div$ with $c$ irreducible components all in $\lvert L^m \rvert$, the image of every non constant morphism $\Lsav \to \Prv$ intersects $\Div$.
\end{thm}

\section{Base loci of the logarithmic cotangent bundle}\label{section:loci}

Let $k$ be a field. We call \emph{$k$-variety} an integral separated scheme of finite type over $\Spec(k)$. If $\Qpv$ is a smooth $k$-variety, a \emph{log compactification} of $\Qpv$ is the data $(\Prv, \Div)$ of a proper smooth $k$-variety $\Prv$ and a normal crossing divisor $\Div \subset \Prv$ such that $\Qpv$ identifies with the open subset $\Prv \setminus \Div$. A smooth $k$-variety is \emph{(projectively) log compactifiable} if it admits a (projective) log compactification.

\subsection{Absolutely logarithmic symmetric forms}

In the following subsection, we will have to consider the graded $k$-algebra of logarithmic symmetric forms on a log compactification of a log compactifiable smooth $k$-variety. This graded $k$-algebra is actually independent of the choice of the log compactification, and Abramovich generalized its definition to every $k$-variety in \cite[Definition 1]{abramovich1994subvarieties}. We give here a reformulation of Abramovich's definition in the smooth case.

\bigskip

Let $K$ be a separable function field over $k$. A divisorial valuation of $K/k$ is a valuation which vanishes of $K/k$ whose residual field has transcendence degree $d - 1$ over $k$, with $d$ the transcendence degree of $K$ over $k$. A valuation is divisorial if and only if it is the valuation defined by a prime divisor of a normal $k$-variety with $K$ as function field.

\bigskip

Let $\nu$ be a divisorial valuation of $K/k$. Let $A$ be the valuation ring $\{x \in K, \nu(x) \geq 0\}$, $\ma$ its maximal ideal $\{x \in K, \nu(x) \geq 1\}$, $\ma^{-1}$ the $A$-submodule $\{x \in K, \nu(x) \geq -1\}$ of $K$, $\kappa := A/\ma$ the residual field of $A$. The canonical maps $A \to \kappa$ and $A \to K$ induce morphisms of $A$-modules $\pi: \Omega_{A/k}^1 \to \Omega_{\kappa/k}^1$ and $f: \Omega_{A/k}^1 \to \Omega_{K/k}^1$. We denote by $F_\nu^1$ (resp. $L_\nu^1$) the $A$-submodule $\Img f$ (resp. $\ma^{-1}f(\ker \pi)$) of $\Omega_{K/k}^1$. Remark that $F_\nu^1 \subset L_\nu^1$. The following proposition is easy to prove.

\begin{prop}\label{prop:lienabr}
	A form $\omega \in \Omega_{K/k}^1$ is in $L_\nu^1$ if and only if there exist $r \in \N$, $u \in \ma$, $f_0, \dots, f_r, x_1, \dots, x_r \in A$ such that $\omega = f_0\frac{\dif u}{u} + \sum_{i=1}^{r} f_i \dif x_i$.
\end{prop}

We denote by $F_\nu^\bullet = \bigoplus_{m \in \N} F_\nu^m$ (resp. $L_\nu^\bullet = \bigoplus_{m \in \N} L_\nu^m$) the free symmetric $A$-algebra generated by $F_\nu^1$ (resp. $L_\nu^1$) with its canonical graduation. The inclusions of $F_\nu^1$ (resp. $L_\nu^1$) in $\Omega_{K/k}^1$ induces a morphism $\iota_F$ (resp. $\iota_L$) of graded $A$-algebras from $F_\nu^\bullet$ (resp. $L_\nu^\bullet$) to $D_K^\bullet := \bigoplus_{m \in \N} \Sym_K^m \Omega_{K/k}^1$. This last graded $A$-algebra is a graded algebra of polynomials over $K$. The following lemma is easy to prove.

\begin{lem}\label{lem:fnl}
	\begin{enumerate}
		\item The $A$-modules $F_\nu^1$ and $L_\nu^1$ are free and finite.
		\item The graded $A$-algebras $F_\nu^\bullet$ and $L_\nu^\bullet$ are graded algebras of polynomials over $A$.
		\item The morphisms $\iota_F$ and $\iota_L$ are injective.
	\end{enumerate}
\end{lem}

Let $\Qpv$ be a smooth $k$-variety. Let $K$ be its functions field. We denote by $D_\Qpv^\bullet$ the graded $k$-subalgebra $\bigoplus_{m \in \N} H^0(\Qpv, \Sym^m \Omega_{\Qpv}^1)$ of $D_K^\bullet$. As every symmetric differential defined in codimension 1 on $\Qpv$ extends as a global symmetric differential form on $\Qpv$, $D_\Qpv^\bullet$ identify with the intersection of the $F_\nu^\bullet$'s over all divisorial valuations $\nu$ of $K/k$ centered in $\Qpv$.

\bigskip

We define the graded $k$-algebra $L_\Qpv^\bullet$ as the intersection of $D_\Qpv^\bullet$ with the $L_\nu^\bullet$'s for every divisorial valuation $\nu$ of $K/k$. Proposition \ref{prop:lienabr} shows that $L_\Qpv^\bullet$ is the algebra of \emph{absolutely logarithmic} symmetric differential forms defined by Abramovich in \cite[Definition 1]{abramovich1994subvarieties}. Then we get the following properties.

\begin{prop}\cite[Remark p.46]{abramovich1994subvarieties}\label{prop:abreq}
	If $(\Prv, \Div)$ is a log compactification of $\Qpv$, then $L_\Qpv^\bullet$ is the graded $k$-algebra $\bigoplus_{m \in \N} H^0(\Prv, \Sym^m \Omega_{\Prv}^1(\log \Div))$.
\end{prop}

Let $\Qpv, \Qpv'$ be smooth $k$-varieties, let $f: \Qpv \to \Qpv'$ be a morphism. Then pull-back of differential forms induces a morphism of graded $k$-algebras $f^*: D_{\Qpv'}^\bullet \to D_\Qpv^\bullet$.

\begin{prop}\cite[Lemma 5]{abramovich1994subvarieties}\label{prop:functform}
	Suppose $\Qpv'$ is log compactifiable. Then $f^*L_{\Qpv'}^\bullet \subset L_\Qpv^\bullet$.
\end{prop}

\begin{prop}
	The integral domain $L_\Qpv^\bullet$ is integrally closed.
\end{prop}

\begin{proof}
	It is the intersection of some integral domains $F_\nu^\bullet$'s and $L_\nu^\bullet$'s that are integrally closed by Lemma \ref{lem:fnl}.
\end{proof}

\begin{prop}\label{prop:covform}
	Let $\Qpv, \Base$ be smooth $k$-varieties. Let $f: \Qpv \to \Base$ be a finite étale covering map. Then $L_\Qpv^\bullet$ is the integral closure of $f^*L_\Base^\bullet$ in $D_\Qpv^\bullet$.
\end{prop}

\begin{proof}
	As $f^*: D_\Base^\bullet \to D_\Qpv^\bullet$ is injective, we consider $D_\Base^\bullet$ as a graded $k$-subalgebra of $D_\Qpv^\bullet$. Let $K$ (resp. $K'$) be the function field of $\Base$ (resp. $\Qpv$). Then $f$ induces a finite separable extension of fields $K'/K$. Hence, every divisorial valuation of $K/k$ is the restriction of a divisorial valuation of $K'/k$. Moreover, if $\nu'$ is a divisorial valuation of $K'/k$ and $\nu$ is its restriction to $K$, then $\nu$ is divisorial and $ L_{\nu}^\bullet = L_{\nu'}^\bullet \cap D_K^\bullet$. Therefore $L_\Base^\bullet = L_\Qpv^\bullet \cap D_\Base^\bullet$ and, with the previous proposition, the integral closure of $L_\Base^\bullet$ in $D_\Qpv^\bullet$ is included in $L_\Qpv^\bullet$.
	
	To prove the converse, we can suppose that $f$ is Galois. Indeed, in any case, if $\Qpv' \to \Base$ is a Galois finite étale covering map which dominates $\Qpv$ and if $L_{\Qpv'}^\bullet$ is the integral closure of $L_\Base^\bullet$ in $D_{\Qpv'}^\bullet$, then the previous inclusion gives $L_\Qpv^\bullet \subset L_{\Qpv'}^\bullet$ which is integral over $L_\Base^\bullet$.
	
	In the case where $f$ is Galois, $L_\Qpv^\bullet$ is stable under the action of the Galois group, hence the minimal polynomial over $\Frac(D_\Base^\bullet)$ of the elements of $L_\Qpv^\bullet$ has coefficients in $L_\Qpv^\bullet \cap D_\Base^\bullet = L_\Base^\bullet$.
\end{proof}

\subsection{Vanishing loci of the logarithmic symmetric forms}

Consider a smooth $k$-variety $\Qpv$ and the projective bundle $\Proj \Omega_{\Qpv}^1 \to \Qpv$ (with Grothendieck's convention: the fiber at a point $x$ is the set of hyperplanes of $\Omega_{\Prv, x}^1$, or equivalently the set of lines of $T_{\Prv, x}$). Then the vanishing locus of a homogeneous element of $L_\Qpv^\bullet$ defines a closed subset of $\Proj \Omega_{\Qpv}^1$ and a closed subset of $T_\Qpv$. Let $m \in \Ns$. We denote by $\PmBL(\Qpv) \subset \Proj \Omega_{\Qpv}^1$ and $\TmBL(\Qpv) \subset T_\Qpv$ the intersection of the vanishing loci of the elements of $L_\Qpv^m$. Remark that $\PmBL$ and $\TmBL$ carry the same information. We denote $\PBL(\Qpv) := \bigcap_{m \in \Ns} \PmBL(\Qpv)$ and $\TBL(\Qpv) := \bigcap_{m \in \Ns} \TmBL(\Qpv)$. Remark that there exists $m \in \Ns$ such that $\PBL(\Qpv) = \PmBL(\Qpv)$ and $\TBL(\Qpv) = \TmBL(\Qpv)$.

\bigskip

Proposition \ref{prop:abreq} implies that if $\Qpv$ admits a log compactification $(\Prv, \Div)$, considering the projective bundle $\pi : \Proj \Omega_{\Prv}^1(\log \Div) \to \Prv$, $\Taut$ the tautological line bundle over $\Proj \Omega_{\Prv}^1(\log \Div)$ and $\Stru(m)$ its tensor power $\Taut^{\otimes m}$ for $m \in \Ns$, then $\PmBL(\Qpv)$ (resp. $\PBL(\Qpv)$) is the intersection of the base locus of $\Stru(m)$ (resp. stable base locus of $\Taut$) with $\Proj \Omega_{\Qpv}^1$.

\bigskip

The partial functoriality of $L^\bullet$ (Proposition \ref{prop:functform}) gives the partial functoriality of base loci.

\begin{prop}\label{prop:funct}
	Let $f: \Qpv \to \Qpv'$ be a morphism between smooth $k$-varieties with $\Qpv'$ log compactifiable. Then for every $m\in\Ns$, $\dif f(\TmBL(\Qpv)) \subset \TmBL(\Qpv')$, and $\dif f(\TBL(\Qpv)) \subset \TBL(\Qpv')$.
\end{prop}

\begin{defi}
	A closed subset $C$ of a $\Proj \Omega_{\Qpv}^1$ is \emph{linearly non-degenerate} if for each closed point $x \in \Qpv$, the fiber $C_x$ is not included in a hyperplane of $\Proj \Omega_{\Qpv,x}^1$.
\end{defi}

\begin{prop}\label{prop:cste}
	Let $f: \Qpv \to \Qpv'$ be a morphism between smooth $k$-varieties with $\Qpv'$ log compactifiable. Suppose the corestriction $f: \Qpv \to \overline{\Img f}$ is separable, there exists $m \in \N$ such that $\PmBL(\Qpv)$ is linearly non degenerate and $\PmBL(\Qpv')$ is empty. Hence $f$ is constant. The same statement holds with $\PBL(\Qpv)$ instead of $\PmBL(\Qpv)$ and $\PBL(\Qpv')$ instead of $\PmBL(\Qpv')$.
\end{prop}

\begin{proof}
	Thanks to Proposition \ref{prop:funct} and the second and third assumptions, $\dif f$ is zero. Hence, thanks to the first assumption, $f$ is constant.
\end{proof}

\begin{prop}\label{prop:covsbl}
	Let $f: \Qpv \to \Base$ be a finite étale covering map of degree $n \in \Ns$ between smooth $k$-varieties. Let $f_{\Proj}$ be the map $\Proj \Omega_{\Qpv}^1 \to \Proj \Omega_{\Base}^1$ induced by $f$. Let $m \in \Ns$. Then $f_{\Proj}^{-1}(\BL_{nm}^{\Proj}(\Base)) \subset \PmBL(\Qpv) \subset f_{\Proj}^{-1}(\PmBL(\Base))$. In particular $\PBL(\Qpv) = f_{\Proj}^{-1}(\PBL(\Base)$.
\end{prop}

\begin{proof}
	 The inclusion $\PmBL(\Qpv) \subset f_{\Proj}^{-1}(\PmBL(\Base))$ is immediate with Proposition \ref{prop:covform}. The same proposition implies that the norm $N_{\Qpv/\Base}(\alpha)$ of an element $\alpha$ of $L_\Qpv^m$ for the extension $\Frac(D_\Qpv^\bullet)/\Frac(D_\Base^\bullet)$ is in $L_\Base^{nm}$. For a smooth $k$-variety $\Prv$, denote by $B_{\Prv}(\beta)$ the vanishing locus of a $\beta \in L_\Prv^\bullet$ in $\Proj \Omega_\Prv^1$. Let $f': \Qpv' \to \Base$ be a Galois finite étale covering map with a morphism of $\Base$-coverings $g: \Qpv' \to \Qpv$. Then $B_{\Qpv'}(\alpha) = g_{\Proj}^{-1}(B_{\Qpv}(\alpha))$. Moreover, the vanishing locus in $\Base$ of the norm $N_{\Qpv'/\Base}(\alpha)$ of $\alpha$ for the extension $\Frac(D_{\Qpv'}^\bullet)/\Frac(D_\Base^\bullet)$ is $f'_{\Proj}(B_{\Qpv'}(\alpha))$ because $N_{\Qpv'/\Base}(\alpha)$ is a product of the conjugates of $\alpha$. Hence $B_\Base(N_{\Qpv'/\Base}(\alpha)) = f_{\Proj}(B_\Qpv(\alpha))$. As $N_{\Qpv'/\Base}(\alpha)$ is a power of $N_{\Qpv/\Base}(\alpha)$, we have $B_\Base(N_{\Qpv/\Base}(\alpha)) = f_{\Proj}(B_\Qpv(\alpha))$. Considering the intersection for all $\alpha \in L_\Qpv^m$, we get $f_{\Proj}^{-1}(\BL_{nm}^{\Proj}(\Base)) \subset \PmBL(\Qpv)$.
\end{proof}

\begin{prop}\label{prop:ext}
	Let $K/k$ be a field extension. Let $\Qpv$ be a log compactifiable smooth $k$-variety. Then for every $m \in \Ns$, $\PmBL(\Qpv \otimes_k K) = \PmBL(\Qpv) \otimes_k K$ and $\PBL(\Qpv \otimes_k K) = \PBL(\Qpv) \otimes_k K$.
\end{prop}

\begin{proof}
	Immediate consequence of \cite[\href{https://stacks.math.columbia.edu/tag/02KH}{Tag 02KH}]{stacks-project}.
\end{proof}

\begin{prop}\label{prop:spe}
	Let $\Prv$ be a smooth and proper scheme over a non empty irreducible scheme $S$ with an étale morphism $S \to \Spec(\Z)$. Let $\Qpv$ be the complement of a normal crossing divisor $\Div$ in $\Prv$. Let $\eta \in S$ be the generic point. Let $m \in \Ns$.
	\begin{itemize}
		\item If $\PmBL(\Qpv_\eta)$ is empty, then for all but finitely many closed points $p \in S$, $\PmBL(\Qpv_p)$ is empty. The same statement holds with $\PBL(\Qpv_\eta)$ instead of $\PmBL(\Qpv_\eta)$ and $\PBL(\Qpv_p)$ instead of $\PmBL(\Qpv_p)$.
		\item If $\PmBL(\Qpv_\eta)$ is linearly non-degenerate, then for all but finitely many (depending on $m$) closed points $p \in S$, $\PmBL(\Qpv_p)$ is linearly non-degenerate.
	\end{itemize}
\end{prop}

\begin{proof}
	Consider the vector bundle $\Sym^m \Omega_{\Prv}^1(\log \Div)$ on $\Prv$. Let $b$ be the morphism $\Prv \to S$. Thanks to Grauert theorem, up to shrink $S$, $b_*\Sym^m\Omega_{\Prv}^1(\log \Div)$ is free and for every $s \in S$, $b_*\Sym^m\Omega_{\Prv}^1(\log \Div)_s = H^0(\Prv_s, \Sym^m\Omega_{\Prv}^1(\log \Div)_{\vert \Prv_s}) = H^0(\Prv_s, \Sym^m\Omega_{\Prv_s}^1(\log \Div_s))$. In particular, if we denote by $Z$ the vanishing locus of $H^0(S, b_*\Sym^m\Omega_{\Prv}^1(\log \Div))$ in $\Proj \Omega_{\Qpv}^1$, then for every $s \in S$, $Z_s = \PmBL(\Qpv_s)$.
	
	Suppose $\PmBL(\Qpv_\eta) = \emptyset$, hence $Z$ is disjoint of the fiber at $\eta$. By definition $Z$ is closed in $\Proj \Omega_{\Qpv}^1$, hence by Chevalley theorem, $b(Z)$ is constructible. As it does not contain the generic point of $S$, it is a finite set of closed points. For the case of $\PBL$, notice that it is empty if and only if one of the $\PmBL$ is empty.
	
	Suppose $\PmBL(\Qpv_\eta)$ is linearly non-degenerate. Let $\di$ be the relative dimension of $\Prv$ over $S$. Consider $P$ the fiber product over $\Qpv$: $P := \prod_{i=1}^{\di} \Proj \Omega_{\Qpv}^1$. Let $C \subset P$ be the open subset of maximal rank $\di$-uples of lines. Let $Z^\di := \prod_{i=1}^{\di} Z \subset P$: then $Z^\di \cap C$ is a locally closed subset of $P$, let $Q$ be its projection on $\Qpv$. Chevalley theorem gives that $Q$ is constructible, hence that $b(\Qpv \setminus Q)$ is constructible. By assumption, it does not contain the generic point of $S$. Hence it is a finite set of closed points.
\end{proof}

\subsection{Augmented base locus of a line bundle}

In this subsection, we suppose $k = \C$. As the characteristic of $\C$ is zero, every (quasiprojective) smooth complex algebraic variety $\Qpv$ is (projectively) log compactifiable. Moreover, if $f: \Qpv \to \Qpv'$ is a morphism between two quasiprojective smooth complex algebraic varieties, then there exists a projective log compactification $\Prv$ (resp. $\Prv'$) of $\Qpv$ (resp. $\Qpv'$) such that $f$ extends as a morphism $\Prv \to \Prv'$.

\bigskip

The augmented base locus was first defined in \cite{ein2006asymptotic}, in continuation of \cite{nakamaye2000stable}. Let $\LB$ and $\Amp$ be line bundles over a projective smooth complex algebraic variety $\Prv$ with $\Amp$ ample. Then, one can easily check that $\bigcap_{m \in \Ns} \StBL(\LB^m \otimes \Amp^{-1})$ does not depend on the choice of $\Amp$. So we can define:

\begin{defi}
	The \emph{augmented base locus} of $\LB$ is:
	\[
	\AugBL(\VB) := \bigcap_{m \in \Ns} \StBL(\LB^m \otimes \Amp^{-1})
	\]
	where $\Amp$ is any ample line bundle.
\end{defi}

The augmented base locus of a nef line bundle can be determined by the computation of its exceptional locus.

\begin{defi}
	Let $\LB$ be a nef line bundle over $\Prv$. Its \emph{exceptional locus}, denoted by $\Exc(\LB)$, is the union of the irreducible closed subsets $Z \subset V$ of positive dimension such that the intersection number $\LB^{\dim Z} \cdot Z$ is $0$.
\end{defi}

Then, the augmented base locus can be computed thanks to the following theorem.

\begin{thm}\cite[Theorem 1.1]{nakamaye2000stable}\label{thm:cas}
	If $\LB$ is nef, then $\AugBL(\LB) = \Exc(\LB)$.
\end{thm}

We also have a the following useful characterization of $\AugBL$.

\begin{thm}\cite{boucksom2014augmented}\label{thm:bcl}
	\begin{enumerate}
		\item The complement of $\AugBL(\LB)$ in $\Prv$ is the maximum for inclusion among Zariski open subsets $U$ of $\Prv \setminus \StBL(\LB)$ such that for $k \in \Ns$ divisible enough, the $k$-th Kodaira map defined by $\LB$ is unramified on $U$.
		\item In the previous statement, we can replace ``unramified'' by ``an isomorphism on its image''.
	\end{enumerate}
\end{thm}

The characterization \ref{thm:bcl}.2 is the form under which Mok mentioned the augmented base locus in \cite[Proposition 4, p.262]{mok1989metric}.

\begin{cor}\label{cor:augsimple}
	For $m \in \Ns$ divisible enough, for every $x \in \Prv \setminus \StBL(\LB)$, we have $x \notin \AugBL(\LB)$ if and only if for every $v \in T_{\Prv, x}$, there exists a global section $\sigma$ of $\LB^{\otimes k}$ which vanishes at $x$ but not on $v$.
\end{cor}

\subsection{Augmented base locus of the logarithmic cotangent bundle}

In this subsection, we suppose $k = \C$.

\begin{defi}
	Let $\Qpv$ be a quasiprojective smooth complex algebraic variety. Let $(\Prv, \Div)$ be a projective log compactification of $\Qpv$. We denote by $\PABL(\Qpv) \subset \Proj \Omega_{\Prv}^1$ the intersection of the augmented base locus of $\Taut$ with $\Proj \Omega_{\Prv}^1$.
\end{defi}

As the Kodaira maps associated with $H^0(\Prv, \Sym^m \Omega_{\Prv}^1(\log \Div)) = L_\Qpv^m$ on $\Proj \Omega_{\Qpv}^1$ do not depend on the choice of the compactification, Theorem \ref{thm:bcl} gives us the fact that $\PABL(\Qpv)$ does not depend on the compactification either. Moreover, Propositions \ref{prop:funct}, \ref{prop:functform} and Corollary \ref{cor:augsimple} give the following functionality property.

\begin{prop}\label{prop:functaug}
	Let $\Qpv$ and $\Qpv'$ be quasiprojective smooth complex algebraic varieties, let $f : \Qpv \to \Qpv'$ be a morphism. Let $R \subset \Qpv$ be the ramification locus of $f$. Let $f_{\Proj}: \Proj \Omega_{\Qpv \setminus R}^1 \to \Proj \Omega_{\Qpv'}^1$ be the morphism induced by $f$. Then $f_{\Proj}(\PABL(\Qpv) \cap \Proj \Omega_{\Qpv \setminus R}^1) \subset \PABL(\Qpv')$.
\end{prop}

\begin{prop}\label{prop:covaug}
	Let $f: \Qpv \to \Base$ be a finite étale covering map between quasiprojective smooth complex algebraic varieties. Let $f_{\Proj}$ be the map $\Proj \Omega_{\Qpv}^1 \to \Proj \Omega_{\Base}^1$ induced by $f$. Let $m \in \Ns$. Then $\PABL(\Qpv) = f_{\Proj}^{-1}(\PABL(\Base))$.
\end{prop}

\begin{proof}
	The inclusion $\PABL(\Qpv) \subset f_{\Proj}^{-1}(\PABL(\Base))$ is given by the previous proposition. To prove the converse, we can suppose that $f$ is Galois. Indeed, in any case, let $f': \Qpv' \to \Base$ be a Galois finite étale covering map with a morphism of $\Base$-coverings $g: \Qpv' \to \Qpv$ and suppose $\PABL(\Qpv') = {f'_{\Proj}}^{-1}(\PABL(\Base))$. Then $f_{\Proj}^{-1}(\PABL(\Base)) = g_{\Proj}({f'_{\Proj}}^{-1}(\PABL(\Base))) = g_{\Proj}(\PABL(\Qpv')) \subset \PABL(\Qpv)$ with the previous inclusion.
	
	Hence, we suppose $f$ Galois. Let $x \in \Proj \Omega_{\Qpv}^1 \setminus \PABL(\Qpv)$ and $v \in T_{\Qpv, x} \cong T_{\Base, f(x)}$. Thanks to Corollary \ref{cor:augsimple} and Proposition \ref{thm:bcl}, there exist $m \in \Ns$, a form in $L_\Qpv^m$ which vanishes at $x$ but not on $v$ and, for each conjugate $x'$ of $x$ under the action of the Galois group, a form in $L_\Qpv^m$ which vanishes at $x$ but not at $x'$. Hence there exists a form $\sigma$ in $L_\Qpv^m$ which realizes all this conditions at the same time. Hence $N_{\Qpv/\Base}(\sigma)$ (notation of the proof of Proposition \ref{prop:covsbl}) vanishes at $x$ but not on $\dif f(v)$. Hence, thanks to Proposition \ref{prop:covsbl} and Corollary \ref{cor:augsimple}, we have $f_{\Proj}^{-1}(\PABL(\Base)) \subset \PABL(\Qpv)$.
\end{proof}

\section{Irreducible bounded symmetric domains}\label{section:lsav}

We start with a recall of some classical facts about bounded symmetric domains, and we set some notations. The proofs of those facts can be found for instance in \cite{helgason2001differential} or in the four first chapters of \cite{mok1989metric}.

\bigskip

Let $\di \in \N$, let $\Dom$ be an irreducible bounded symmetric domain of $\C^{\di}$. It is simply connected. The group $\Aut(\Dom)$ of biholomorphisms of $\Dom$ has a natural structure of simple real Lie group with finitely many connected components. Let $\Group \subset \Aut(\Dom)$ be its neutral connected component. The maximal compact Lie subgroups of $\Group$ are the stabilizers of points of $\Dom$. Fix $x \in \Dom$, let $\Isot$ be the isotropy group of $x$ (the stabilizer of $x$ in $\Group$), so as $\Dom = \Group / \Isot$. The center $S$ of $\Isot$ is a real Lie group. In the natural representation of $\Isot$ over the holomorphic tangent space $T_{\Dom, x}$, $S$ acts by complex homotheties, inducing a real Lie groups isomorphism $\iota$ between $S$ and the circle group $\U$. We denote $z := \iota^{-1}(\im)$.

\bigskip

Consider the smooth representation of $S$ over the real Lie algebra $\glie$ associated with $\Group$ given by adjunction. We decompose the complexified Lie algebra $\glie^{\C}$ into isotypic components:
\[\glie^\C = \llie^\C \oplus \mlie^+ \oplus \mlie^-\]
where:
\begin{itemize}
	\item $S$ acts trivially on $\llie^\C$,
	\item $S$ acts on $\mlie^+$ by homotheties with rate given by the isomorphism $\iota: S \cong \U$,
	\item $S$ acts on $\mlie^-$ by homotheties with rate given by the conjugation of the isomorphism $\iota$.
\end{itemize}
The component $\llie^\C$ is a complex Lie subalgebra of $\glie^\C$. It is the complexification of the real Lie subalgebra $\llie \subset \glie$ assoicated with the real Lie subgroup $\Isot \subset \Group$. The components $\mlie^+$ and $\mlie^-$ are stable under the representation of $\Isot$ over $\glie^\C$ given by adjunction. The complexification of the differential at the neutral element of the map $\Group \to \Dom$ induces an isomorphism of complex representations of $\Isot$ between $\mlie^+$ (resp. $\mlie^-$) and the complex vector space of holomorphic (resp. antiholomorphic) tangent vectors of $\Dom$ at $x$.

\bigskip

Let $\Group^\C$ be the adjoint Lie group of $\glie^\C$. Set $\plie := \llie^\C \oplus \mlie^-$: it is a complex Lie subalgebra of $\glie^\C$. Let $\Pgroup$ be the connected complex Lie subgroup of $\Group^\C$ assoicated with $\plie$. Let $\Dual := \Group^\C / \Pgroup$. Then, the inclusion $\Group \subset \Group^\C$ induces an open embedding $\Dom \hookrightarrow \Dual$, equivariant with respect to the action of $\Group$.

\bigskip

Consider $\pi: \Fib \to \Dom$ the projectivization of the holomorphic cotangent bundle of $\Dom$. Projectivization defines the tautological line bundle $\Taut$ on $\Fib$. The domain $\Dom$ admits a canonical Hermitian metric invariant under the action of $\Aut(\Dom)$, the Bergman metric (see \cite[\S1, p.55]{mok2002characterization}). We denote it $\Came$. It induces an Hermitian metric $\came$ on $\Stru(-1)$.

\bigskip

We can express the curvature of $\came$ using the Riemannian curvature tensor $\Ricu: T_{\Dom}^{\otimes 4} \to \R$ of $\Came$. Let $x \in \Dom, v \in T_{\Dom, x}$ be a unit vector, we denote $y := (x,[v])$ its image in $\Fib$. Consider some local holomorphic coordinates $(z^1, \dots, z^\di)$ in an open neighborhood of $x$. By applying a biholomorphism of $\Dom$ and a linear automorphism of $\C^\di$, we can suppose that they are centered at $x$, that the matrix $\left ( \Came \left ( \frac{\partial}{\partial z^i}, \frac{\partial}{\partial z^j} \right )_{ij} \right )$ is the identity at $x$, and that $v = \frac{\partial}{\partial z^{\di}}$. If we denote $u^i := \frac{\dif z^i}{\dif z^{\di}}$ ($1 \leq i < \di$), then $(z^1, \dots, z^{\di}, u^1, \dots, u^{\di-1})$ are local holomorphic coordinates of $\Proj \Omega_{\Dom}^1$ centered in $y$. Finally, we denote $R_{i\bar{j}k\bar{l}} := \Ricu \left ( \frac{\partial}{\partial z^i}, \frac{\partial}{\partial z^j}, \frac{\partial}{\partial z^k}, \frac{\partial}{\partial z^l} \right )$ ($1 \leq i, j, k, l \leq \di$). Then, \cite[Proposition 1, p.38]{mok1989metric} gives the first Chern form of $\came$:
\begin{equation}\label{eq:curv}
\Chern(\Taut, \came)(y) = \frac{1}{2\pi} \sum_{1 \leq k < \di} \dif u^k \wedge \dif \bar{u}^k - \frac{1}{2 \pi} \sum_{1 \leq i, j \leq \di} R_{\di\bar{\di}i\bar{j}} \dif z^i \wedge \dif \bar{z}^i.
\end{equation}

\bigskip

In this article, we are interested in smooth complex algebraic varieties $\Lsav$ whose universal covering is biholomorphic to an irreducible bounded symmetric domain $\Dom \subset \C^\di$. In this case, the fundamental group $\Latt'$ of $\Lsav$ is a subgroup of $\Aut(\Dom)$. Such varieties are closely linked to locally symmetric varieties in the following sense.

\begin{defi}\label{defi:lsv}
	A \emph{locally symmetric variety} is the quotient of an irreducible bounded domain $\Dom$ by a torsion free lattice $\Latt$ of the neutral component $\Group$ of the real Lie group of biholomorphisms $\Aut(\Dom)$.
\end{defi}

As isotropy groups are compact, the action of $\Latt$ over $\Dom$ is free and properly discontinuous, hence a locally symmetric variety has a natural structure of complex manifold. Compactification theorems of Baily, Borel \cite{baily1966compactification} and Mok \cite[Main Theorem]{mok2011projective} give the following proposition.

\begin{prop}\label{prop:uniclsav}
	A locally symmetric variety admits a unique structure of complex algebraic variety compatible with its structure of complex manifold. Moreover, this structure is quasiprojective.
\end{prop}

This has the following consequences.

\begin{prop}\label{prop:lsav}
	Let $\Lsav$ be a smooth complex algebraic variety whose universal covering is biholomorphic to an irreducible bounded symmetric domain. Then $\Lsav$ has a finite étale covering which is a locally symmetric variety.
\end{prop}

\begin{cor}
	A smooth complex algebraic variety whose universal covering is isomorphic to an irreducible bounded symmetric domain is quasiprojective.
\end{cor}

Let $\Lsav$ be a locally symmetric variety whose fundamental group $\Latt$ is an arithmetic and neat lattice (see \cite[\S 17.1]{borel1969introduction} for the definition of neat). A nonempty family of log compactifications of $\Lsav$ with an explicit description is given in \cite{ash2010smooth}. A compactification in this family will be called \emph{toroidal}.

\subsection{Characteristic subvarieties}\label{subsection:char}

The proofs of the results of this subsection can be found in \cite[Chapter 10, Appendices III \& IV]{mok1989metric}. Let $r \in \N$ be the rank of $\Dom$. Let $x \in \Dom$, let $\Isot$ be the isotropy group of $x$. We equip $\Dom$ with the Bergman metric $\Came$. The starting point is the polydisk theorem.

\begin{thm}\label{thm:polyd}
	There is a totally geodesic complex submanifold $D \subset \Dom$ biholomorphically isometric to the Poincaré polydisk $\Disk^r$ such that $\displaystyle \Dom = \bigcup_{\gamma \in \Isot} \gamma \cdot D$ and $\displaystyle T_x \Dom = \bigcup_{\gamma \in \Isot} \dif \gamma \cdot T_x D$.
\end{thm}

\begin{prop}
	Let $v \in T_x \Dom$. Let $D, D'$ be two totally geodesic complex submanifolds of $\Dom$ with biholomorphic isometries $f: D \to \Disk^r$ and $f': D' \to \Disk^r$, such that $x \in D \cap D'$ and $v \in T_xD \cap T_xD'$. Then $\dif f_x (v)$ has as many non zero coordinates as $\dif f'_x (v)$.
\end{prop}

\begin{defi}
	Let $v \in T_x \Dom$. Let $D$ be a totally geodesic complex submanifold of $\Dom$ with a biholomorphic isometry $f: D \to \Disk^r$ such that $x \in D$ and $v \in T_xD$. The \emph{rank} of $v$ is the number of non zero coordinates of $\dif f_x (v)$.
\end{defi}

The rank exists thanks to Theorem \ref{thm:polyd} and is unique thanks to the previous proposition. As the rank is invariant under the action of $\C^*$, we can set the following definition.

\begin{defi}
	Let $0 \leq k \leq r$. We call \emph{$k$-th characteristic subvariety at $x$} the subset $\Charac_k^x \subset \Proj \Omega_{\Dom, x}^1$ of lines of vectors of rank at most $k$. We call \emph{$k$-th characteristic bundle of $\Dom$} the union $\displaystyle\Charac_k(\Dom) := \bigcup_{x \in \Dom} \Charac_k^x \subset \Fib$.
\end{defi}

The characteristic subvarieties are Zariski closed complex subvarieties of $\Proj \Omega_{\Dom, x}^1$. The characteristic bundles are closed holomorphic subbundles of $\Fib$ (over $\Dom$). Notice that $\Charac_0^x = \emptyset$ and $\Charac_0(\Dom) = \emptyset$.

\bigskip

Fix $1 \leq k \leq r$. We can describe characteristic subvarieties in terms of group actions: $\Charac_k^x$ is stable under the action of $K_x$ over $\Proj \Omega_{\Dom, x}^1$. Moreover, this action induces an action of the complexified group $K_x^{\C}$ over $\Proj \Omega_{\Dom, x}^1$. Hence, if we denote $\Charac_k^{\reg,x} := \Charac_k^x \setminus \Charac_{k-1}^x$ and $\Charac_k(\Dom)^{\reg} := \Charac_k(\Dom) \setminus \Charac_{k-1}(\Dom)$, according to the first section of \cite{mok2002characterization}, the set $\Charac_k^{\reg,x}$ is an orbit of $\Proj \Omega_{\Dom, x}^1$ under the action of $K_x^{\C}$. Moreover, $\Charac_k(\Dom)^{\reg}$ is the intersection of $\Fib$ with an orbit of $\Proj \Omega_{\Dual}^1$ under the action of $\Group^\C$. In particular, $\Charac_k^{\reg,x}$ (resp. $\Charac_k(\Dom)^{\reg}$) is a smooth open subset of $\Charac_k^x$ (resp. $\Charac_k(\Dom)$). Remark that it is dense (by definition of characteristic subvarieties).

\bigskip

We can make a link between characteristic subvarieties and curvature. For a holomorphic vector $v$ tangent to $\Dom$ at $x$, we denote by $\Null_v$ the complex vector space of holomorphic vectors $u$ such that the antilinear form $w \mapsto \Ricu(v, \bar{v}, u, \bar{w})$ is zero. There is a unique decreasing application $\nul : [\![1, r]\!] \to [\![0, \di]\!]$ such that for each $v \in T_x\Dom$ of rank $k$, $\dim \Null_v$ is $\nul(k)$. This number determines the dimension of characteristic subvarieties.

\begin{prop}
	We have $\dim \Charac_k^x = \di - \nul(k)$ and $\dim \Charac_k(\Dom) = 2\di - \nul(k) - 1$.
\end{prop}

Using the equation (\ref{eq:curv}), we see that for $y = (x, [v]) \in \Fib$ of rank $k$, the projection $\pi : \Fib \to \Dom$ induces an isomorphism between the kernel of $\Chern(\Taut, \came)(y)$ and $\Null_v$. Moreover, the kernel of $\Chern(\Taut, \came)(y)$ is tangent to $\Charac_k(\Dom)$. This proves the following proposition.

\begin{prop}\label{prop:rank}
	The pull-back of $\Chern(\Taut, \came)(y)$ on $\Charac_k(\Dom)^\reg$ has rank $2\di - 2\nul(k)-1$.
\end{prop}

Finally, an explicit computation gives the following result on dimensions.

\begin{prop}\label{prop:codim}
	If $1 \leq k < r - 1$, $\Charac_k(\Dom)$ has codimension at least $2$ in $\dim \Charac_{k+1}(\Dom)$.
\end{prop}

Let $\Lsav$ be a smooth complex algebraic variety whose universal covering is isomorphic to $\Dom$. Let $\Latt \subset \Aut(\Dom)$ be its fundamental group. Let $0 \leq k \leq r$. Remark that $\Charac_k(\Dom)$ and $\Charac_k(\Dom)^{\reg}$ are stable under the action of $\Latt$. We denote $\Charac_k := \Latt \setminus \Charac_k(\Dom) \subset \QprB$ and $\Charac_k^{\reg} := \Latt \setminus \Charac_k(\Dom)^{\reg}$. We will call $\Charac_k$ a \emph{characteristic subvarieties} of $\Lsav$ (though they are subvarieties of $\QprB$). We will denote by $\Charac$ the characteristic subvariety $\Charac_{r-1}$.

\subsection{Algebraicity of characteristic subvarieties}

In this subsection, we prove that if $\Lsav$ is a smooth complex algebraic variety whose universal covering is isomorphic to $\Dom$, the characteristic subvarieties of $\Lsav$ are Zariski closed subsets of $\QprB$. We use notions from o-minimal geometry (see \cite{van1998tame}). We denote by $\Ralg$ the real o-minimal structure of semialgebraic sets and by $\Ranexp$ the real o-minimal structure generated by the globally subanalytic sets and by the exponential (see \cite{van1994real}).

\bigskip

Notice that $\Group, \Group^{\C}, \Dom$ and $\Dual$ have a canonical structure of $\Ralg$-definable manifold, for which the inclusion $\Group \subset \Group^{\C}$ and the actions of $\Group$ on $\Dom$ and of $\Group^{\C}$ on $\Dual$ are definable (\cite[Lemma 2.1]{bakker2020tame}). Moreover, the definable structure on $\Dom$ induces a $\Ralg$-definable structure on $\Fib$.

\begin{lem}
	For every $0 \leq k \leq r$, the subset $\Charac_k(\Dom) \subset \Fib$ is $\Ralg$-definable.
\end{lem}

\begin{proof}
	Choose a base point $x \in \Dom$. The actions of $\Group$ and $\Group^\C$ on $x \in \Dual$ gives the following commutative square.
	\begin{center}
		\begin{tikzcd}
		\Group \arrow[r, hook] \arrow[d, two heads] & \Group^{\C} \arrow[d, two heads] \\
		\Dom \arrow[r, hook]                        & \Dual                           
		\end{tikzcd}
	\end{center}
	The top, left and right maps of the square are $\Ralg$-definable. Hence, the bottom map is definable too. Thus the induced map $\Fib \to \Proj \Omega_{\Dual}^1$ is $\Ralg$-definable. In consequences, $\Charac_k(\Dom)$ is definable because it is the inverse image by a definable map of a finite union of orbits under the definable action of $\Group^{\C}$ over $\Proj \Omega_{\Dual}^1$.
\end{proof}

\begin{prop}\label{prop:alg}
	For every $0 \leq k \leq r$, $\Charac_k$ is a closed algebraic subvariety of $\QprB$.
\end{prop}

\begin{proof}
	Using Proposition \ref{prop:lsav}, we can suppose that $\Lsav$ is a locally symmetric variety. Let $\Latt \subset \Group$ be its fundamental group. If $\Dom$ has at most rank $1$, the proposition is obvious. Then suppose the rank of $\Dom$ is at least $2$. First \cite[Theorem 1.9]{klingler2016hyperbolic} gives us a $\Ralg$-definable fundamental set $F$ of $\Dom$ for the action of $\Latt$ such that the surjection $F \to \Lsav$ is $\Ranexp$-definable. Let $F'$ be the inverse image of $F$ in $\Fib$, then the surjection $F' \to \QprB$ is $\Ranexp$-definable too. Moreover $F' \cap \Charac_k(\Dom)$ is $\Ralg$-definable, and its image in $\QprB$ is $\Charac_k$. Hence $\Charac_k$ is $\Ranexp$-definable. Moreover $\Charac_k$ is a closed analytic subset, hence the definable Chow theorem \cite[Corollary 4.5.]{peterzil2009complex} gives the proposition.
\end{proof}

Alternatively, we can prove the algebraicity of $\Charac_k$ using its characterization in \cite[Definition 3.1 \& Theorem 3.3]{sheng2010polarized}.

\section{Growth moderation near singularities}\label{section:chernweil}

To prove Theorem \ref{thm:main}, we will need a theorem of Kollár \cite[Theorem 5.20]{kollar1987subadditivity}. In this section, we rewrite and sometime change the definitions, statements and proofs of \cite[Section 5]{kollar1987subadditivity}, first to complete some arguments that we were not able to follow (in particular the proof of \cite[Theorem 5.20]{kollar1987subadditivity} seems incomplete), secondly because we will use in Section \ref{section:main} notions of growth moderation for differential forms near a normal crossing divisor closed to the ones introduced by Kollár.

\bigskip

The differences between this section and \cite[Section 5]{kollar1987subadditivity} are the following. First, we change the definition of ``almost bounded'' so that almost bounded forms form an algebra, that is used to prove \cite[Theorem 5.20]{kollar1987subadditivity} (Theorem \ref{thm:chern}). Secondly, we detail the proof of \cite[Proposition 5.16.(ii)]{kollar1987subadditivity} (Proposition \ref{prop:int}.2). Thirdly, we introduce the notion of ``quasipositive'' to justify the integrability in \cite[Corollary 5.17]{kollar1987subadditivity} (Proposition \ref{prop:intnul}). Finally, we allow a constant $\Const$ in the definition of ``nearly bounded'' and we introduce the notion of ``(almost) quasibounded''. This last point is not useful for the study of PVHS, but will be important in Section \ref{section:main}.

\bigskip

Let $r \in (0, e^{-1})$. We consider the disk and the punctured disk of radius $r$:
\[\Disr := \{z \in \C, \lvert z \rvert < r\},\]
\[\Pdis := \Disr \setminus \{0\}.\]
We denote $\Dipo$ the divisor $\{z_1\dots z_\di = 0\}$ of $\Disr^\di$. For $1 \leq i, j \leq \di, i \neq j$ and $\Const, \epsilon \in \R_{>0}$, we denote:
\[\Thin_{i, j}^{\Const,\epsilon} := \{ z \in \Ppdi, \lvert z_j \rvert \leq \exp(- \Const \lvert z_i \rvert^{- \epsilon}) \},\] \[\Thin_i^{\Const, \epsilon} := \bigcup_{j \neq i} \Thin_{i, j}^{\Const,\epsilon}.\]
For $0 < \epsilon < \epsilon'$ and $0 < \Const < \Const'$, we have $\Thin_{i, j}^{\Const',\epsilon'} \subset \Thin_{i, j}^{\Const,\epsilon}$ and $\Thin_i^{\Const',\epsilon'} \subset \Thin_i^{\Const,\epsilon}$.

\subsection{Functions with moderate growth}

\begin{defi}\label{defi:nb}
	Let $f$ be a smooth complex-valued function over the punctured polydisk $\Ppdi$.
	\begin{itemize}
		\item $f$ has \emph{at most log growth} if there exist $A \in \R_{\geq 0}$ and $k \in \N$ such that for every $z = (z_i) \in \Ppdi$:
		\[\lvert f(z) \rvert \leq A \prod_{i=1}^\di (- \ln \lvert z_i \rvert)^k.\]
		\item $f$ has \emph{at most log-log growth} if there is $B \in \R_{\geq 0}$ such that for every $z = (z_i) \in \Ppdi$:
		\[\lvert f(z) \rvert \leq B \sum_{i=1}^\di \ln(-\ln \lvert z_i \rvert).\]
		\item For $\Const, \epsilon \in \R_{>0}$, $D \in \R_{\geq 0}$ and $p \in \N$, $f$ is $(\Const, \epsilon, D, p)$-\emph{nearly bounded} if for each $z \in \Ppdi$ one of the following cases occurs:
		\begin{itemize}
			\item $\lvert f(z) \rvert \leq D$,
			\item there exists $1 \leq i \leq \di$ such that $z \in \Thin_i^{\Const,\epsilon}$ and $\lvert f(z) \rvert \leq D (- \ln \lvert z_i \rvert)^p$.
		\end{itemize}
		\item $f$ is \emph{nearly bounded} if there exist $\Const, \epsilon \in \R_{>0}$, $D \in \R_{\geq 0}$ and $p \in \N$ such that $f$ is $(\Const, \epsilon, D, p)$-nearly bounded.
		\item $f$ is \emph{quasibounded} if it is the product of a function with at most log-log growth by a nearly bounded function.
		\item $f$ is \emph{integrable} if it is integrable for the Lebesgue measure.
		\item $f$ is \emph{quasipositive} if it is the sum of a nonnegative real-valued function and an integrable function.
	\end{itemize}
\end{defi}

The functions that are continuous on $\Disr^\di$ has at most log-log growth and are nearly bounded. Functions with at most log-log growth and nearly bounded functions are quasibounded. Quasibounded functions have at most log growth and are integrable. Integrable functions are quasipositive. If $f$ has values in $\R_{>0}$ and $f$ and $f^{-1}$ have at most log growth, then $\ln \circ f$ have at most log-log growth. The following proposition is an easy consequence of the definitions.

\begin{prop}\label{prop:algebra}
	Nearly bounded functions form a complex algebra. Quasibounded functions form a module over this algebra.
\end{prop}

\subsection{Differential forms with moderate growth on the punctured polydisk}

We denote by $\Cinf(\Ppdi)$ the ring of smooth complex-valued functions over $\Ppdi$. We denote by $\Smofo(\Ppdi)$ the graded $\Cinf(\Ppdi)$-algebra of smooth complex differential forms over $\Ppdi$. The \emph{Poincaré basis} is the $\Cinf(\Ppdi)$-basis of $\Smofo(\Ppdi)$ whose elements are the ordered wedge products of elements of $\left (\frac{\dif z_1}{\lvert z_1 \rvert (- \ln \lvert z_1 \rvert)}, \dots, \frac{\dif z_\di}{\lvert z_\di \rvert (- \ln \lvert z_\di \rvert)}, \frac{\dif \bar{z_1}}{\lvert z_1 \rvert (- \ln \lvert z_1 \rvert)}, \dots, \frac{\dif \bar{z_\di}}{\lvert z_\di \rvert (- \ln \lvert z_\di \rvert)} \right )$.

\begin{defi}
	A form of $\Smofo(\Ppdi)$ is \emph{nearly bounded} (resp. \emph{quasibounded}) if its coefficients in the Poincaré basis are nearly bounded (resp. quasibounded) functions.
\end{defi}

As a corollary of Proposition \ref{prop:algebra}, nearly bounded forms form a graded complex subalgebra of $\Smofo(\Ppdi)$.

\begin{defi}
	A form of $\Smofo(\Ppdi)$ of degree $2 \di$ is \emph{integrable} (resp. \emph{semipositive}, \emph{quasipositive}) if it is the product of an integrable (resp. nonnegative real-valued, quasipositive) function by $\im \dif z_1 \wedge \dif \bar{z_1} \wedge \dots \wedge \im \dif z_\di \wedge \dif \bar{z_\di}$.
\end{defi}

The integral of a quasipositive form of degree $2 \di$ is a well defined element of $\C \cup \{+\infty\}$. It is finite if and only if the form is integrable.

\begin{prop}\label{prop:int}
	\begin{enumerate}
		\item A quasibounded form of degree $2\di$ is integrable.
		\item If a quasibounded form $\alpha$ has degree $2\di-1$, is supported on a compact of $\Disr^\di$ and $\dif \alpha$ is quasipositive, then $\dif \alpha$ is integrable and $\int_{\Ppdi} \dif \alpha = 0$.
	\end{enumerate}
\end{prop}

\begin{proof}
	\begin{enumerate}
		\item We have to check that if $g$ is a nearly bounded function, then
		\[\ln(- \ln \lvert z_1 \rvert) g(z) \prod_{i=1}^\di \lvert z_i \rvert^{-2} (- \ln \rvert z_i \lvert)^{-2}\]
		is integrable.
		
		Let $\Const, \epsilon \in \R_{>0}, D \in \R_{\geq 0}$ and $p \in \N$ be such that $g$ is $(\Const, \epsilon, D, p)$-nearly bounded. As $\lvert z \rvert^{-2} (- \ln \rvert z \lvert)^{-2}$ and $\lvert z \rvert^{-2} (- \ln \rvert z \lvert)^{-2} \ln(- \ln \lvert z \rvert)$ are integrable over $\Disr$, $\alpha$ is integrable on $\Disr^\di \setminus \bigcup_i \Thin_i^{\Const,\epsilon}$.
		
		Moreover, for $1 \leq i, j \leq \di, i \neq j, i \neq 1, j \neq 1$:
		\begin{align*}
		&\int_{\Thin_{i, j}^{\Const, \epsilon}} \ln(-\ln \lvert z_1 \rvert) (- \ln \lvert z_i \rvert)^p \prod_{k=1}^\di \lvert z_k \rvert^{-2} (- \ln \rvert z_k \lvert)^{-2} \dif z_1 \wedge \dif \bar{z_1} \wedge \dots \wedge \dif z_\di \wedge \dif \bar{z_\di} \\
		= & E \int_{\Disr} \min\{(-\ln r)^{-1}, \Const^{-1} \lvert z_i \rvert^\epsilon\} \lvert z_i \rvert^{-2} (- \ln \lvert z_i \rvert)^{p-2} \dif z_i \wedge \dif \bar{z_i} \\
		\leq &E C^{-1} \int_{\Disr} \lvert z_i \rvert^{-(2-\epsilon)} (- \ln \lvert z_i \rvert)^p \dif z_i \wedge \dif \bar{z_i} \\
		< & +\infty
		\end{align*}
		where $E \in \R_{>0}$ is a constant. The case $i = 1$ simply add a factor $\ln(-\ln \lvert z_i \rvert)$ to the last integral, that does not change the integrability.
		
		Finally, we have to treat the case $j=1$. Let $1 < i \leq \di$:
		\begin{align*}
		&\int_{\Thin_{i, 1}^{\Const, \epsilon}} \ln(-\ln \lvert z_1 \rvert) (- \ln \lvert z_i \rvert)^p \prod_{k=1}^\di \lvert z_k \rvert^{-2} (- \ln \rvert z_k \lvert)^{-2} \dif z_1 \wedge \dif \bar{z_1} \wedge \dots \wedge \dif z_\di \wedge \dif \bar{z_\di} \\
		= & F \int_{\Disr} \min\{(-\ln r)^{-1} (1 + \ln(-\ln r)), (\Const^{-1} \lvert z_i \rvert^\epsilon) (1 + \ln(\Const \lvert z_i \rvert^{-\epsilon}))\} \\
		& \lvert z_i \rvert^{-2} (- \ln \lvert z_i \rvert)^{p-2} \dif z_i \wedge \dif \bar{z_i} \\
		\leq & FC^{-1} \int_{\Disr} \lvert z_i \rvert^{-(2-\epsilon)} (- \ln \lvert z_i \rvert)^{p-2} (1 + \ln(\Const) + \epsilon (-\ln(\lvert z_i \rvert))) \dif z_i \wedge \dif \bar{z_i} \\
		< & +\infty.
		\end{align*}
		
		\item Let $\Disk$ be the disk of radius $1$, and $\Disk^*$ the punctured disk of radius $1$. Consider the following function:
		\[\application{\eta}{(\Disk^*)^\di}{\R}{z}{\sum_{i=1}^\di \ln(-\ln \lvert z_i \rvert)}\]
		with differential:
		\[\omega := \dif \eta = \sum_{i=1}^\di \frac{\dif \lvert z_i \rvert}{\lvert z_i \rvert \ln \lvert z_i \rvert}. \]
		Consider the Poincaré Hermitian metric over $(\Disk^*)^\di$ given by:
		\[\sum_{i=1}^\di \frac{\lvert \dif z_i \rvert^2}{\lvert z_i \rvert^2 (\ln \lvert z_i \rvert)^2}. \]
		We set $v$ the real vector field over $(\Disk^*)^\di$ which is the dual of $\omega$ with respect to the Poincaré metric. Then $v$ is complete. Indeed, as its norm is constant (of value $\sqrt{d}$ at every point), integral curves of $v$ with bounded domain are bounded for the Poincaré metric. As this one is complete, an integral curve with bounded domain is relatively compact. Hence maximal integral curves are defined on $\R$.
		
		
		Consider an integral curve $x: \R \to (\Disk^*)^\di$ for $v$. Then, for $t \in \R$:
		\[(\eta \circ x)'(t) = \omega(v(x(t))) = d.\]
		Hence $\eta \circ x: \R \to \R$ is affine, in particular it is a diffeomorphism.
		
		For $t \in \R$, let $U_t := \{\eta = t\}$. We just proved that the flow of $v$ gives a diffeomorphism $\phi: \R \times U_0 \cong (\Disk^*)^\di$, with $\phi(\{t\} \times U_0) = U_t$.
		
		Now, consider the homothety $h:\Disk^\di \to \Disr^\di$ of rate $r$. Then $h_* \omega$ has Poincaré growth. Recall that $\alpha$ is a quasibounded form of degree $2\di-1$ supported on a compact of $\Disr^\di$ such that $\dif \alpha$ is quasipositive. Hence, thanks to Proposition \ref{prop:algebra}, $h_* \omega \wedge \alpha$ is quasibounded. Thanks to the first point, it is integrable. Hence, if we denote $\alpha' := h^*\alpha$, $\omega \wedge \alpha'$ is integrable. Moreover, $\phi^*\omega = \phi^*\dif \eta = \dif \phi^* \eta = \dif t$. Hence:
		\begin{equation}\label{eq:finite}
		\int_0^{+\infty} \left ( \int_{U_t} \alpha' \right ) \dif t = \int_{(\Disk^*)^\di} \omega \wedge \alpha'
		< +\infty.
		\end{equation}
		Moreover, if we denote $Z_t := \bigsqcup_{s<t} U_s$:
		\begin{equation}
		\int_{(\Disk^*)^\di} \dif \alpha' = \lim_{t \to +\infty} \int_{Z_t} \dif \alpha' = \lim_{t \to +\infty} \int_{U_t} \alpha'.
		\end{equation}
		and (\ref{eq:finite}) implies $\lim_{t \to +\infty} \int_{U_t} \alpha' = 0$.
	\end{enumerate}
\end{proof}

\subsection{Differential forms with moderate growth on a complex manifold}

Let $\Anman$ be an complex manifold. Let $\Div$ be a normal crossing divisor of $\Anman$. Let us extend the notions of the previous subsection to this case.

\begin{defi}\label{defi:almbo}
	\begin{itemize}
		\item An \emph{admissible chart} on $(\Anman, \Div)$ is a holomorphic open embedding $f: \Disr^\di \to \Anman$ such that $f^{-1}(\Div) \subset \Dipo$.
		\item A smooth differential form $\alpha$ over $\Anman \setminus \Div$ is \emph{locally nearly bounded} (resp. \emph{locally quasibounded}) if for each $r \in (0, e^{-1})$ and each admissible chart $\Disr^\di \to \Anman$ on $(\Anman, \Div)$, there exists $r' \in (0,r)$ such that the pull-back of $\alpha$ on $(\Disk_{r'}^*)^\di$ is nearly bounded (resp. quasibounded).
		\item A \emph{permissible blow-up} on $(\Anman, \Div)$ is the data of a complex manifold $\Anman'$, a normal crossing divisor $\Div'$ and a blow-up $b: \Anman' \to \Anman$ of center $Z \subset \Anman$ such that $b^{-1}(\Div) = \Div'$ and for each $x \in Z$, there is an admissible chart $f: \Disr^\di \to (\Anman, \Div)$ such that $f(0) = x$ and $f^{-1}(Z)$ is the intersection of some of the $\{ z_i = 0 \}$.
		\item Let $E \subset \Anman$ be a divisor. A \emph{desingularization} of $E$ in $(\Anman, \Div)$ is a composition of permissible blow-ups $g: (\Anman', \Div') \to (\Anman, \Div)$ such that $g^{-1}(E) \cup \Div'$ has normal crossings.
		\item A smooth differential form $\alpha$ on $\Ppdi$ is \emph{pre-almost (quasi)bounded} if there exists a divisor $E \subset \Disr^\di$ such that the pull-back of $\alpha$ by every desingularization of $E$ in $(\Disr^\di, \Dipo)$ is locally nearly bounded (resp. locally quasibounded).
		\item A smooth differential form $\alpha$ on $\Anman \setminus \Div$ is \emph{almost (quasi)bounded} if for each $r \in (0, e^{-1})$ and for each admissible chart $f: \Disr^\di \to (\Anman, \Div)$, there exists $r' \in (0, r)$ such that the restriction of $f^*\alpha$ on $\Disk_{r'}^\di$ is pre-almost (quasi)bounded.
		\item A smooth differential form $\alpha$ of degree $2 \di$ on $\Anman \setminus \Div$ is \emph{locally integrable} (resp. \emph{semipositive}, \emph{locally quasipositive}), if for every $x \in \Anman$, there is an admissible chart $f$ on $(\Anman, \Div)$ such that $f(x) = 0$ and $f^*\alpha$ is integrable (resp. semipositive, quasipositive).
	\end{itemize}
\end{defi}

The integral of a locally quasipositive form of degree $2 \di$ that is supported on a compact of $\Anman$ is a well defined element of $\C \cup \{+\infty\}$. It is finite if and only if the form is locally integrable. Thanks to \cite[Theorem 1.1.9]{temkin2018functorial}, every divisor has a desingularization in $(\Anman, \Div)$.

\begin{prop}\label{prop:algebra2}
	Almost bounded forms over $\Anman \setminus \Div$ form a complex graded algebra. Almost quasibounded forms over $\Anman \setminus \Div$ form a graded module over this graded algebra.
\end{prop}

\begin{proof}
	It is a consequence of Proposition \ref{prop:algebra} and of the fact that a desingularization of the union of two divisors is a desingularisation of each of those divisors.
\end{proof}

The following proposition is a direct consequence of Proposition \ref{prop:int}.1.

\begin{prop}\label{prop:intalm}
	An almost quasibounded (resp. locally quasibounded, pre-almost quasibounded) form of degree $2\di$ is locally integrable.
\end{prop}

We also have an analog of the second point of Proposition \ref{prop:int}. First, Propositions \ref{prop:algebra}, \ref{prop:int}.1, \ref{prop:algebra2} and \ref{prop:intalm} used with Leibniz formula give us the following lemma.

\begin{lem}\label{lem:rho}
	Let $\alpha$ be an almost quasibounded (resp. quasibounded, locally quasibounded, pre-almost quasibounded) form of degree $2\di-1$ on $\Anman \setminus \Div$ (resp. $\Ppdi$, $\Anman \setminus \Div$, $\Ppdi$) such that $\dif \alpha$ is (locally) quasipositive. Let $\rho$ be a smooth positive real function on $\Anman$ (resp. $\Disr^\di$, $\Anman$, $\Disr^\di$). Then $\rho \alpha$ is almost quasibounded (resp. quasibounded, locally quasibounded, pre-almost quasibounded) and $\dif (\rho \alpha)$ is (locally) quasipositive.
\end{lem}

This lemma allows us to use partitions of unity to deduce the following proposition from Proposition \ref{prop:int}.

\begin{prop}\label{prop:intnul}
	Let $\alpha$ be an almost quasibounded form of degree $2\di-1$ on $\Anman \setminus \Div$ such that $\dif \alpha$ is locally quasipositive. Then $\dif \alpha$ is locally integrable. If, moreover, $\alpha$ is supported on a compact of $\Anman$, then $\int_{\Anman \setminus \Div} \dif \alpha = 0$.
\end{prop}

\begin{cor}\label{cor:clcur}
	An almost quasibounded form $\alpha$ on $\Anman \setminus \Div$ defines a current on $\Anman$. If $\alpha$ is closed, then this current is closed.
\end{cor}

\section{Real polarized variations of Hodge structure and singularities}\label{section:pvhs}

Let $\Anman$ be a complex manifold. In the following, we identify a finite locally free $\Stru_{\Anman}$-module with the associated holomorphic vector bundle. In particular, a holomorphic vector subbundle corresponds to a $\Stru_{\Anman}$-submodule which is locally a direct factor, and a morphism between two holomorphic vector bundles corresponds to a morphism of $\Stru_{\Anman}$-modules.

\bigskip

We abbreviate ``real polarized variation of Hodge structure'' to ``PVHS''. When we denote by $(\Locsys_{\R}, \Pol, \Filho^\bullet)$ a PVHS, $\Locsys_{\R}$ is the underlying real local system, $\Pol$ is the polarization form, and $\Filho^\bullet$ is the Hodge filtration. We denote by $\Locsys_{\Stru}$ the holomorphic vector bundle associated with $\Locsys_{\R}$. The Higgs map, which is a morphism of vector bundles $\Filho^p / \Filho^{p+1} \to \Omega_{\Anman}^1 \otimes \Filho^{p-1} / \Filho^p$, is denoted by $\theta$. We denote by $\bigoplus_{p+q=k} \Locsys^{p,q}$ the Hodge decomposition of $\Locsys_{\Stru}$ and by $\home$ the Hodge metric over $\Locsys_{\Stru}$. We will use the same notation for the Hermitian metric induced on holomorphic vector subquotients of $\Locsys_{\Stru}$.

\bigskip

Now, suppose $\Anman$ is the complement of a normal crossing divisor $\Div$ in a complex manifold $\Anman'$. Let $\Locsys$ be a complex local system on $\Anman$ with unipotent monodromy around $\Div$. We denote by $\ext{\Locsys}$ the Deligne extension of the holomorphic vector bundle associated with $\Locsys$ on $\Anman'$ (see \cite[Proposition II.5.2]{deligne2006equations}). If $(\Locsys_{\R}, \Pol, \Filho^{\bullet})$ is a PVHS over $\Anman$ with unipotent monodromy around $\Div$, we denote by $\ext{\Filho}^\bullet$ the canonical extension of $\Filho^{\bullet}$ on $\ext{\Locsys}$ given by Schmid's Nilpotent Orbit Theorem (\cite[Theorem 4.12]{schmid1973variation}). The definition of the Deligne extension allows to extend the flat connection as a logarithmic connection $\ext{\Filho}^p \to \Omega_{\Anman'}^1(\log \Div) \otimes \ext{\Filho}^{p-1}$ (for $p \in \Z$) and the Higgs map as a morphism of holomorphic vector bundles over $\Anman'$: $\theta: \ext{\Filho}^p / \ext{\Filho}^{p+1} \to \Omega_{\Anman'}^1(\log \Div) \otimes \ext{\Filho}^{p-1} / \ext{\Filho}^p$.

\subsection{Singularities of the Hodge metric}

We will need later the following simplified form of \cite[Theorem (5.21)]{cattani1986degeneration}.

\begin{thm}\label{thm:homelog}
	Let $(\Locsys_{\R}, \Pol, \Filho^\bullet)$ be a PVHS over $\Ppdi$ with unipotent monodromy. Let $s$ be a non vanishing section of $\ext{\Locsys}$. Then $\home(s,s)$ and $\home(s,s)^{-1}$ have at most log growth.
\end{thm}

In the following statement, Koll{\'a}r's definition of almost bounded is different from ours, but one can easily check that the proof works for our definition.

\begin{lem}\cite[Proposition 5.15]{kollar1987subadditivity}
	Let $\VB$ be a holomorphic vector line subbundle of the canonical extension of a PVHS over $\Anman \setminus \Div$ with unipotent monodromy around $\Div$, let $\home$ be its Hodge metric. Then the first Chern form $\Chern(\VB, \home)$ is almost bounded, and for every trivializing section $s$ of $\det(\VB)$ on an open subset $U$ of $\Anman$, the connection form $\partial \ln \lVert s \rVert_{\home}^2$ is almost bounded on $U \setminus \Div$.
\end{lem}

The two following statements of Koll{\'a}r were originally stated with ``nearly bounded'' instead of ``almost bounded''. However, his proof goes by reduction to the case of a line subbundle for which only the almost boundedness is known. Moreover, Koll{\'a}r's proof uses the fact that the class of forms under consideration form an algebra. This is clear for the class of nearly bounded forms, but not for the class of almost bounded forms in the original definition. This is why we changed the definition of almost bounded forms, so that they form an algebra too.

\begin{prop}\cite[Remark 5.19]{kollar1987subadditivity}\label{prop:homealm}
	Let $\VB$ be a holomorphic vector subquotient of the canonical extension of a PVHS over $\Anman \setminus \Div$ with unipotent monodromy around $\Div$, let $\home$ be its Hodge metric. Then the first Chern form $\Chern(\VB, \home)$ is almost bounded, and for every trivializing section $s$ of $\det(\VB)$ on an open subset $U$ of $\Anman$, the connection form $\partial \ln \lVert s \rVert_{\home}^2$ is almost bounded on $U \setminus \Div$.
\end{prop}

\begin{thm}\cite[Theorem 5.20]{kollar1987subadditivity}\label{thm:chern}
	Let $\VB_1, \dots, \VB_n$ be holomorphic vector bundle subquotients of PVHSs over $\Anman \setminus \Div$ with unipotent monodromy around $\Div$, let $\home_1, \dots, \home_n$ be their Hodge metrics, let $k_1, \dots, k_n \in \Ns$. Let $P \in \C[X_1, \dots, X_n]$ be a polynomial in $n$ variables. Then $P(C_{k_1}(\VB_1, \home_1), \dots, C_{k_n}(\VB_n, \home_n))$ is almost bounded and defines a closed current on $\Anman$ whose de Rham cohomology class is $P(c_{k_1}(\VB_1), \dots, c_{k_n}(\VB_n))$.
\end{thm}

\begin{proof}
	Same proof as \cite[Theorem 5.20]{kollar1987subadditivity}, using Corollary \ref{cor:clcur}.
\end{proof}

The last result we will need is the following theorem.

\begin{thm}\cite[Theorem 1.8]{brunebarbe2018symmetric}\label{thm:nef}
	Let $\Prv$ be a smooth complex projective variety, $\Div \subset \Prv$ a normal crossing divisor, $(\Locsys_{\R}, \Pol, \Filho^\bullet)$ be a PVHS over $\Prv \setminus \Div$ with unipotent monodromy around $\Div$, $p \in \Z$ and $L$ a holomorphic vector subbundle of $\ext{\Filho}^p/\ext{\Filho}^{p+1}$ of rank $1$ in the kernel of the Higgs map $\theta$. Then the dual of $L$ is nef.
\end{thm}

\subsection{Case of locally symmetric varieties}

We use the notations of Section \ref{section:lsav}. We denote $\Lsav$ a locally symmetric variety and $\Latt \subset \Group$ its fundamental group (see Definition \ref{defi:lsv}). A way to construct a PVHS of weight $0$ over $\Lsav$ is given in \cite[\S 4]{zucker1981locally}. We describe briefly this construction here, and we refer the reader to the loc. cit. article for details. Using the same notations as in Section \ref{section:lsav}, we choose an irreducible real Lie group representation $\rho: G \to \GL(E)$ of $\Group$.

\bigskip

First, we define a PVHS over $\Dom$. We set $\Locsys_{\R}$ the constant local system with fiber $E$ over $\Dom$. The polarization is given by the admissible inner product. Let $y \in \Dom$. As $z$ (defined in Section \ref{section:lsav}) is in the center of $\Isot$, $gzg^{-1}$ does not depend on the choice of $g \in \Group$ as long as $gx = y$. We define $\Filho^{\bullet}$ so as at $y$, $\Locsys^{p,q}$ are eigenspaces of $gzg^{-1}$. This construction descends to the quotient by $\Latt$.

\bigskip

The monodromy of the local system is given by the representation of $\Latt$ over $E$. If $p+q = 0$, $\sigma$ is an element of $\Filho_x^p / \Filho_x^{p+1} \cong \Locsys_x^{p,q}$ and $v \in T_x \Lsav \cong \mlie^+$ is a holomorphic tangent vector, then $\theta(v)\sigma = \dif\rho(v)\sigma$. The PVHS over $\Dom$ is invariant under the action of $\Group$. If $\Latt$ is arithmetic and neat and $(\Comp, \Div)$ is a toroidal compactification of $\Lsav$, then for $p \leq q$, the metric induced by the Hodge metric on $\ext{\Filho}^p / \ext{\Filho}^q$ is ``good'' in the sense of \cite{mumford1977hirzebruch}. In particular, if $f$ is an admissible chart on $(\Comp, \Div)$ and $\sigma$ is a section of $f^*\ext{\Filho}^p / \ext{\Filho}^q$, then $\lVert \sigma \rVert_{\home}$ has at most log growth.

\begin{prop}\label{prop:exvhs}
	If $\Latt$ is arithmetic, then there exist a finite étale covering map $\kappa: \Fico \to \Lsav$, a PVHS $(\Locsys_{\R}, \Pol, \Filho^\bullet)$ over $\Fico$, a log compactification $(\Comp, \Div)$ of $\Fico$ such that $\Locsys$ has unipotent monodromy around $\Div$, and a locally split injective morphism of $\Stru_{\Comp}$-modules $T_{\Comp}(-\log \Div) \to \ext{\Filho}^{-1} / \ext{\Filho}^0$ such that over $\Fico$, the pull-back of the Hodge metric is the metric induced by $\Came$ on $\Fico$, and the image of $T_{\Comp}(-\log \Div)$ in $\ext{\Filho}^{-1} / \ext{\Filho}^0$ is in the kernel of the Higgs map.
\end{prop}

\begin{proof}
	Construct as previously a PVHS $(\Locsys_{\R}^0, \Pol^0, \Filho^{0 \bullet})$ over $\Lsav$ starting from the adjoint representation of $\Group$ over $\glie$. As $\Latt$ is an arithmetic lattice, it has a finite index subgroup whose representation on $\glie$ is integral. This new lattice has a neat finite index subgroup $\Latt'$. We set $\kappa: \Fico := \Latt' \setminus \Dom \to \Lsav = \Latt \setminus \Dom$. We define $(\Locsys_{\R}, \Pol, \Filho^{\bullet})$ as the PVHS on $\Fico$ constructed from the representation of $\Group$ over $\R$-linear endomorphisms $\End(\glie)$ (induced by adjoint representation).
	
	As $\Fico$ is a locally symmetric variety and $\Latt'$ is arithmetic and neat, $\Fico$ admits a toroidal compactification $(\Comp, \Div)$. As $\kappa^* \Locsys^0$ is integral, Borel's theorem \cite[Theorem 6.5]{deligne2006travaux} implies that its monodromy is quasiunipotent around $\Div$. As $\Latt'$ is neat, it is unipotent. As $\Locsys = \End(\kappa^*\Locsys^0)$, its monodromy around $\Div$ is unipotent too.
	
	The Higgs map of $(\kappa^*\Locsys_{\R}^0, \kappa^*\Pol^0, \kappa^*\Filho^{0\bullet})$ induces a morphism of $\Stru_{\Comp}$-modules $\theta: T_{\Comp}(- \log \Div) \to \ext{\Filho}^{-1} / \ext{\Filho}^0$. It is locally split injective over $\Fico$ because for $v \in \mlie^+$, $\theta(v)z = [v,z] = - \im v$. It is also locally split injective over $\Div$ because $T_{\Comp}(- \log \Div)$ is the good extension of $T_{\Fico}$ is the sense of \cite{mumford1977hirzebruch}. Hence the pull-back $\theta^*\home$ of the Hodge metric on $\Locsys$ is an Hermitian metric on $\Fico$. Finally, if we pull back the PVHS on $\Dom$, the Higgs map is equivariant with respect to the action of $\Group$. As $\Group$ acts transitively on $\Dom$ and $\Isot$ acts irreducibly on $T_{\Dom,x}$, $\theta^*\home$ is a multiple of (the metric induced by) $\Came$. Hence we can change $\Pol$ by a multiplicative constant so as over $\Fico$, $\theta^*\home = \Came$. For the proof that the image of $\theta$ is in the kernel of the Higgs map, see \cite[Lemma 3.1]{brunebarbe2018symmetric}.
\end{proof}

\section{Base loci of a locally symmetric variety}\label{section:main}

Now we are ready to prove Theorem \ref{thm:main}. We begin with a technical lemma.

\begin{lem}\label{lem:comp}
	Let $r \in (0, e^{-1})$.
	\begin{itemize}
		\item Let $f$ be a smooth positive real-valued function over $\Ppdi$ such that $f^{-1}$ has at most log growth.
		\item Let $a$ be a holomorphic function over $\Disr^\di$ such that $\{a = 0\} \subset \Dipo$.
	\end{itemize}
	Let $g := \lvert a \rvert^2$. Then there exists $r' \in (0,r)$ such that $\frac{\partial g}{f+g}$ is nearly bounded on $(\Disk_{r'}^*)^\di$.
\end{lem}

\begin{proof}
	Let $(z_i)$ be the coordinates of $\Disr^\di$. Because $\{a = 0\} \subset \Dipo$, there exists nonnegative integers $(n_i) \in \N^\di$ and a non-vanishing holomorphic function $\alpha$ over $\Disr^\di$ such that:
	\begin{equation}\label{eq:deca}
	a = \alpha \prod_{i=1}^\di z_i^{n_i}.
	\end{equation}
	Then:
	\[\frac{\partial g}{g} = \frac{\partial \alpha}{\alpha} + \sum_{i=1}^\di n_i \frac{\partial z_i}{z_i},\]
	hence:
	\begin{equation}
	\frac{\partial g}{f+g} = \frac{g}{f+g}\frac{\partial \alpha}{\alpha} + \sum_{i=1}^\di n_i \frac{g}{f+g}\frac{\partial z_i}{z_i}.\label{eq:decomp}
	\end{equation}
	Consider the first term of (\ref{eq:decomp}). As $f$ and $g$ are positive real-valued on $\Ppdi$, $\frac{g}{f+g}$ is bounded (by $1$) on it. As $\alpha$ is non-vanishing and continuous over $\Disr^\di$ and $\partial \alpha$ is continuous over it, so for each $r' \in (0,r)$, $\frac{\partial \alpha}{\alpha}$ is bounded on $(\Disk_{r'}^*)^\di$. Hence $\frac{g}{f+g}\frac{\partial \alpha}{\alpha}$ is bounded on it too.
	
	So, going back to (\ref{eq:decomp}), it remains to check that for each $i$ such that $n_i > 0$, there exists $r' \in (0,r)$ such that on $(\Disk_{r'}^*)^\di$, $\frac{g}{f+g}\frac{\partial z_i}{z_i}$ is a nearly bounded form. We have:
	\[\frac{g}{f+g}\frac{\partial z_i}{z_i} = \frac{g (- \ln \lvert z_i \rvert)}{f+g} \frac{\partial z_i}{z_i (- \ln \lvert z_i \rvert)}. \]
	As the factor $\frac{\partial z_i}{z_i (- \ln \lvert z_i \rvert)}$ has Poincaré growth, it remains to check that there exists $r' \in (0,r)$ such that on $(\Disk_{r'}^*)^\di$, the positive real-valued function
	\[b := \frac{g (- \ln \lvert z_i \rvert)}{f+g}\]
	is nearly bounded.
	
	To do this, we first need to give the values of the constants $\Const$ and $\epsilon$ that appear in Definition \ref{defi:nb} of nearly bounded functions, and the radius $r' \in (0, r)$ of the polydisk we have to restrict ourselves to. Without loss of generality, we consider the index $i = 1$ and suppose $n_1 > 0$. Because $f^{-1}$ has at most log growth, there are $A \in \R_{>0}, k \in \Ns$ such that for every $z \in \Ppdi$:
	\begin{equation}\label{eq:flog}
	f(z) \geq A \prod_{i=1}^{\di} (- \ln \lvert z_i \rvert)^{-k}.
	\end{equation}
	Let:
	\begin{equation}\label{eq:defep}
	\epsilon := \frac{n_1}{\di k} > 0.
	\end{equation}
	The application:
	\[\application{\xi}{(0,r]}{\R_{>0}}{h}{(-\ln(h) - 1)^{-\frac{1}{\di k}} h^{-\epsilon}}\]
	is continuous, and $\lim\limits_{h \to 0} \xi(h) = +\infty$. Hence $\xi$ is bounded below by a constant $>0$. Moreover, choose any $r'' \in (0, r)$. As $\lvert \alpha(z) \rvert^2$ is positive real-valued continuous on $\Disr^\di$, $\lvert \alpha(z) \rvert^{-\frac{2}{dk}}$ is bounded below by a constant $>0$ on $\Disk_{r''}^\di$. Therefore, we can define the following positive real constant.
	\begin{equation}\label{eq:defcons}
	\Const := \inf_{z \in (\Disk_{r''}^*)^\di} (A^{-1}\lvert \alpha(z) \rvert^2 (-\ln \lvert z_1 \rvert - 1))^{-\frac{1}{\di k}}\lvert z_1 \rvert^{-\epsilon} > 0.
	\end{equation}
	Finally, we choose $r' \in (0, r'')$ small enough so that for every $h \in (0, r')$, $\Const h^{- \epsilon} > - \ln h$.
	
	Now that we have our two constants $\Const$ and $\epsilon$ and our radius $r'$, we can check that $b$ is $(\Const, \epsilon, 1, 1)$-nearly bounded on $\Disk_{r'}^\di$. First, for each $z \in (\Disk_{r'}^*)^\di$, we have $b(z) \leq - \ln \lvert z_1 \rvert$, in particular it is true for $z \in \Thin_1^{\Const, \epsilon}$.
	
	Now, let $z \in (\Disk_{r'}^*)^\di \setminus \Thin_1^{\Const, \epsilon}$, which means that for every $1 < i \leq \di$, $\ln \lvert z_i \rvert > - \Const \lvert z_1 \rvert^{-\epsilon}$. Moreover, we chose $r'$ so that we have $\ln \lvert z_1 \rvert > - \Const \lvert z_1 \rvert^{-\epsilon}$ too. Then, for $1 \leq i \leq n$:
	\begin{align}
	\ln \lvert z_i \rvert &> - \Const \lvert z_1 \rvert^{-\epsilon} \\
	&\geq - \left ( A^{-1}\lvert \alpha(z) \rvert^2(-\ln \lvert z_1 \rvert - 1)\lvert z_1 \rvert^{2n_1} \right )^{-\frac{1}{\di k}} \label{eq:useep} \\
	&\geq - \left ( A^{-1}\lvert \alpha(z) \rvert^2(-\ln \lvert z_1 \rvert - 1) \prod_{j=1}^\di \lvert z_j \rvert^{2n_j} \right )^{-\frac{1}{\di k}}, \label{eq:majgro}
	\end{align}
	using the definition (\ref{eq:defcons}) of $\Const$ and the definition (\ref{eq:defep}) of $\epsilon$ for the line (\ref{eq:useep}), and the fact that for every $1 \leq j \leq n$, $\lvert z_j \rvert < 1$ for the line (\ref{eq:majgro}). Therefore:
	\begin{equation}\label{eq:maj2}
	(- \ln \lvert z_i \rvert)^{-k} > \left ( A^{-1}\lvert \alpha(z) \rvert^2(-\ln \lvert z_1 \rvert - 1) \prod_{j=1}^\di \lvert z_j \rvert^{2n_j} \right )^{\frac{1}{\di}}
	\end{equation}
	Hence:
	\begin{align}
	f(z) &\geq A \prod_{i=1}^\di (-\ln \lvert z_i \rvert)^{-k} \label{eq:reflog} \\
	& > (-\ln \lvert z_1 \rvert - 1) \lvert \alpha(z) \rvert^2 \prod_{i=1}^\di \lvert z_i \rvert^{2n_i} \label{eq:remajgro} \\
	& = (-\ln \lvert z_1 \rvert - 1) g(z), \label{eq:usedeca}
	\end{align}
	using (\ref{eq:flog}) for the line (\ref{eq:reflog}), (\ref{eq:maj2}) for the line (\ref{eq:remajgro}) and (\ref{eq:deca}) for the line (\ref{eq:usedeca}). Eventually, we get $f(z) + g(z) > (- \ln \lvert z_1 \rvert)g(z)$. Hence by definition, $b$ is bounded (by $1$) outside of $\Thin_1^{\Const, \epsilon}$.
\end{proof}

Let us get back to the notations of Subsection \ref{section:lsav}. Moreover, we suppose $\Dom$ of rank at least $2$. We set $\Lsav$ a smooth complex algebraic variety whose universal covering is biholomorphic to $\Dom$. We denote $\pi: \QprB \to \Lsav$ the canonical projection. Let $\Taut$ be the tautological line bundle over $\QprB$, $\Stru(p) := \Taut^{\otimes p}$ its tensor power (for $p \in \Z$), $\came$ the Hermitian metric induced on $\Omun$ by the canonical Kähler metric $\Came$ on $\Lsav$, quotient of the Bergman metric of $\Dom$. Let $\nu$ be a Kähler form on $\QprB$ (for example, the sum of $\Chern(\Stru(-1), \came)$ and of the canonical Kähler form over $\Lsav$). In the following, we will forget to mention $\Omun$ in the notations of first Chern forms $\Chern$.

\bigskip

Let $p \in \Ns$, let $\gsec$ be a global section of $\Stru(p)$ over $\QprB$, suppose that it comes from a global section of $\Sym^p\Omega_{\Comp}^1(\log \Div)$ for some log compactification $(\Comp, \Div)$ of $\Lsav$ (it does not depend on the choice of the compactification). We have to prove that $\gsec$ vanishes on $\Charac$ (recall that $\Charac := \Charac_{r-1}$). Remark that $\vame := (\came^p + \gsec \otimes \bar{\gsec})^{\frac{1}{p}}$ defines an Hermitian metric of seminegative curvature on $\Omun$ (see \cite[Proof of Proposition 2 p.260]{mok1989metric}). Let $u := \ln\left( \frac{\vame}{\came} \right)$. Then $\chri := \partial u$ is the difference between the Chern connections of $\vame$ and $\came$, it is a differential form of type $(1,0)$ defined over $\QprB$ (which does not depend on the choice of local trivializations). Fix a $k$ such that $1 \leq k < r$, let $\incl: \Charac_k^{\reg} \hookrightarrow \QprB$ be the inclusion. Let $\beta := \Chern(\came)^{2\di-2\nul(k)-1} \wedge \nu^{\nul(k)-1}$, with $\nul(k)$ defined in Subsection \ref{subsection:char}.

\begin{lem}
	We have:
	\[\incl^*(\bar{\chri} \wedge \chri \wedge \beta) = 0.\]
\end{lem}

\begin{proof}
	It is obviously sufficient to prove the lemma for a finite étale covering of $\Lsav$. Then, thanks to Proposition \ref{prop:lsav}, we can suppose that $\Lsav$ is a locally symmetric variety. As the rank of $\Dom$ is at least $2$, the superrigidity theorem of Margulis \cite{margulis1991discrete} implies that $\Latt$ is arithmetic. Hence, we can suppose by Proposition \ref{prop:exvhs} that we have a PVHS $(\Locsys_{\R}, \Pol, \Filho^\bullet)$ over $\Lsav$ and a log compactification $(\Comp, \Div)$ of $\Lsav$ such that $\Locsys$ has unipotent monodromy and there is a locally split injective morphism of holomorphic complex vector bundles $T_{\Comp}(-\log \Div) \to \ext{\Filho}^{-1} / \ext{\Filho}^0$ such that over $\Lsav$, the pull-back of the Hodge metric is the metric induced by $\Came$.
	
	We extend $\pi$ as the canonical projection $\PrB \to \Comp$ and $\Taut$ as the tautological line bundle over it. By definition, $\gsec$ is (the restriction of) a global section of $\Stru(p)$ over $\PrB$. We denote $\Div' := \pi^{-1} \Div$, hence $(\PrB, \Div')$ is a log compactification of $\QprB$.
	
	Let $\bar{\Charac_k}$ be the closure of $\Charac_k$ in $\PrB$ for the Hausdorff topology. Thanks to Proposition \ref{prop:alg}, $\bar{\Charac_k}$ is a Zariski closed subvariety of $\PrB$, and its intersection with $\Proj \Omega_{\Lsav}^1$ is $\Charac_k$. Hironaka's desingularization theorem gives us a smooth projective variety $A$ and a proper birational morphism $b: A \to \bar{\Charac_k}$ such that $\Div_1 := b^{-1}(\{\gsec = 0\} \cup \Div')$ is a normal crossing divisor. We now work on $A$ and, in order to lighten notations, we will forget the mention of ``$b^* \incl^*$'' before $\Stru_{\PrB}$-modules and differential forms.
	
	Let us show that $\chri \wedge \beta$ is almost bounded and $u \chri \wedge \beta$ is almost quasibounded. Let $r \in (0, e^{-1})$, let $f:\Disr^\di \to A$ be an admissible map on $(A, \Div_1)$. Let $s$ be a trivializing section of $\Omun$ defined in a neighborhood of $f(0)$. Thanks to Proposition \ref{prop:homealm} (``Connection form of the Hodge metric on a subquotient of a PVHS with unipotent monodromy near a normal crossing divisor is almost bounded''), there exists $r' \in (0,r)$ such that $s$ is defined on $f(\Disk_{r'}^\di)$ and the pull-backs of $\partial \ln \lVert s \rVert_{\came}^2$ and $\Chern(\came)$ on $(\Disk_{r'}^*)^\di$ are pre-almost bounded. There exists a divisor $E$ on $\Disk_{r'}^\di$ such that the pull-backs of $f^*\partial \ln \lVert s \rVert_{\came}^2$ and $f^*\Chern(\came)$ by any desingularization of $E$ are locally nearly bounded. Let $g: (\Anman, \Div_\Anman) \to (\Disk_{r'}^\di, \Dipo)$ be a desingularization of $E$. Let $t \in (0,e^{-1})$, let $f_1: \Disk_{t}^\di \to \Anman$ be an admissible chart on $(\Anman, \Div_\Anman)$. Then, there is $t' \in (0,t)$ such that the restrictions of $f_1^* g^* f^* \partial \ln \lVert s \rVert_{\came}^2$ and $f_1^* g^* f^*\Chern(\came)$ to $(\Disk_{t'}^*)^\di$ are nearly bounded. It remains to show that $f_1^* g^* f^* (\chri \wedge \beta)$ is nearly bounded and $f_1^* g^* f^* (u \chri \wedge \beta)$ is quasibounded on $\Disk_{t'}^\di$. Now, we work on $\Disk_{t'}^\di$, and we forget to mention ``$f_1^* g^* f^*$'' before differential forms.
	
	We have:
	\begin{align*}
	\chri &= \partial \ln \left ( \frac{\lVert s \rVert_{\vame}^2}{\lVert s \rVert_{\came}^2} \right )\\
	&= \frac{1}{p} \partial \ln \left ( \lVert s \rVert_{\came}^{2p} + \lvert \gsec(s^p) \rvert^2 \right ) - \partial \ln \left ( \lVert s \rVert_{\came}^2 \right )
	\end{align*}
	The right term is nearly bounded. Hence, it remains to show that the left term is nearly bounded. We have:
	\begin{equation*}
	\partial \ln (\lVert s \rVert_{\came}^{2p} + \lvert \gsec(s^p) \rvert^2) = \frac{\lVert s \rVert_{\came}^{2p}}{\lVert s \rVert_{\came}^{2p} + \lvert \gsec(s^p) \rvert^2} \partial \ln (\lVert s \rVert_{\came}^{2p}) + \frac{\partial (\lvert \gsec(s^p) \rvert^2)}{\lVert s \rVert_{\came}^{2p} + \lvert \gsec(s^p) \rvert^2}.
	\end{equation*}
	
	The first term is nearly bounded because $\frac{\lVert s \rVert_{\came}^{2p}}{\lVert s \rVert_{\came}^{2p} + \lvert \gsec(s^p) \rvert^2}$ is bounded by $1$ and $\partial \ln \lVert s \rVert_{\came}^2$ is nearly bounded. Moreover, we can pull back the PVHS and apply Theorem \ref{thm:homelog} (``Hodge metric has at most log growth at singularities of PVHS with unipotent monodromy'') to obtain that for each $i \in \Z$, $\lVert s \rVert_{\came}^i$ has at most log growth. Then Lemma \ref{lem:comp} tells that the second term is nearly bounded. Therefore $\chri$ is nearly bounded and, as $\beta$ too, $\chri \wedge \beta$ is nearly bounded. Moreover, as for each $i \in \Z$, $\lVert s \rVert_{\came}^i$ has at most log growth, so $u$ has at most log-log growth. Hence, on $A$, $f^*(\chri \wedge \beta)$ is almost bounded and $f^*(u \chri \wedge \beta)$ is almost quasibounded.
	
	Because of $\dif \chri = \Chern(\vame) - \Chern(\came)$ and Proposition \ref{prop:rank} (``on $\Charac_k^{\reg}$, $\Chern(\came)$ has rank $2\di - 2\nul(k)-1$''), we have $\dif f^*(\chri \wedge \beta) = f^*(\Chern(\vame) \wedge \beta)$, which is a semipositive form. Indeed, as the Bergman metric $\Came$ has a seminegative Griffiths curvature, the formula \ref{eq:curv} implies that the first Chern form $\Chern(\Taut, \came)$ is semipositive. Hence, Proposition \ref{prop:intnul} gives us $\int_{A} \dif f^*(\chri \wedge \beta) = 0$, and by positivity, $\dif f^*(\chri \wedge \beta) = 0$. Hence by Leibniz formula: $\dif f^*(u \chri \wedge \beta) = f^*(\bar{\chri} \wedge \chri \wedge \beta)$. As the first member is almost quasibounded and the second member is clearly semipositive, we obtain in the same way $f^*(\bar{\chri} \wedge \chri \wedge \beta) = 0$. Thus $\incl^*(\bar{\chri} \wedge \chri \wedge \beta) = 0$.
\end{proof}

In the previous proof, we could not simply use the fact that $\Came$ is good (in the sense of \cite{mumford1977hirzebruch}) because this property is not necessarily shared by $\came$ (see \cite[Example 3.9]{i2007generalisations}).

\bigskip

Now that we have the previous lemma, the proof of the inclusion $\Charac_{k-1} \subset \PBL(\Lsav)$ is exactly the same as Mok's one in the compact case. First \cite[Chapter 6 (3.2)]{mok1989metric} gives the following lemma.

\begin{lem}
	The pull-back of $u$ over $\Charac_{r-1}(\Dom) \subset \Fib$ is invariant under the action of $\Group$.
\end{lem}

Then the proof of \cite[Appendix 4 Proposition 4]{mok1989metric} gives the inclusion we want.

\begin{lem}
	We have $\Charac \subset \PBL(\Lsav)$.
\end{lem}

We are now able to prove Theorem \ref{thm:main}, which is the equality $\PBL(\Lsav) = \PABL(\Lsav) = \Charac$.

\begin{proof}[Proof (of Theorem \ref{thm:main})]
	The inclusion $\Charac \subset \PBL(\Lsav)$ is the previous lemma. The inclusion $\PBL(\Lsav) \subset \PABL(\Lsav)$ is true for every smooth complex algebraic variety. It remains to show $\PABL(\Lsav) \subset \Charac$. Thanks to Proposition \ref{prop:covaug}, it is enough to prove the result for a finite étale covering of $\Lsav$. Hence, thanks to Proposition \ref{prop:lsav}, we can suppose that $\Lsav$ is a locally symmetric variety. Hence, thanks to Proposition \ref{prop:exvhs}, we can suppose that we have a PVHS $(\Locsys_{\R}, \Pol, \Filho^\bullet)$ over $\Lsav$ and a log compactification $(\Comp, \Div)$ of $\Lsav$ such that $\Locsys$ has unipotent monodromy and there is a locally split injective morphism of holomorphic complex vector bundles $T_{\Comp}(-\log \Div) \to \ext{\Filho}^{-1} / \ext{\Filho}^0$ such that over $\Lsav$, the pull-back of the Hodge metric is the metric induced by the canonical metric $\Came$. Hence Theorem \ref{thm:nef} gives us the fact that $\Stru_{\Lsav}(1)$ is nef. Hence, thanks to Theorem \ref{thm:cas}, it remains to show that $\Exc(\Taut) \subset \Charac \cup \Div$.
	
	Let $Z \subset \PrB$ be an irreducible closed subset of positive dimension $n$ such that $Z \setminus (\Charac \cup \Div)$ is not empty. We have to check that $\int_Z \chern(\Taut)^n > 0$. Let $Z^{\sing}$ be the closed subset of singular points of $Z$. Consider $f: Z' \to Z$ a birational proper map such that $Z'$ is smooth, $f^{-1} (Z^{\sing} \cup \Div')$ is a normal crossing divisor and $f$ induces an isomorphism over $Z \setminus (Z^{\sing} \cup \Div')$. Pulling back the PVHS $(\Locsys_{\R}, \Pol, \Filho^\bullet)$ to $Z'$, and using Theorem \ref{thm:chern}, we see that it remains to show that $\int_{Z'} \Chern(f^*\Taut, \came)^n > 0$. This results from the fact that $\Chern(\Taut, \came)$ is semipositive on $\QprB$ and positive outside of $\Charac_k$.
\end{proof}

\section{Families of curves}

In this section, we prove Theorems \ref{thm:gen}, \ref{thm:bro} and \ref{thm:mainapp}. We use the definitions of the introduction.

\subsection{The complex case}

Consider an integer $g \geq 2$. Let $\Qpv$ be a quasiprojective smooth complex algebraic variety with a projective log compactification $(\Prv, \Div)$. One can easily check that the ampleness of $\Omega_{\Prv}^1(\log \Div)$ with respect to $\Qpv$ in the sense of \cite{viehweg2001positivity} is equivalent to the emptiness of $\PABL(\Qpv)$, so we have the following proposition.

\begin{prop}\cite[Proposition 2.6]{viehweg2001positivity}
	Suppose there is a family of smooth projective curves of genus $g$ $f: C \to \Qpv$ such that the Kodaira-Spencer map $T_C \to R^1f_*T_{C/\Qpv}$ is a locally split injection. Then $\PABL(\Qpv)$ is empty.
\end{prop}

\begin{lem}\label{lem:vie}
	Let $M$ be a smooth complex algebraic variety with a finite étale covering map $M \to \Mg \otimes \C$. Then $\PABL(M)$ is empty. In particular $\PBL(M)$ is empty.
\end{lem}

\begin{proof}
	According to Proposition \ref{prop:covaug}, the conclusion does not depend on $M$, hence we can chose the moduli space of connected smooth projective curves of genus $g$ with a $3$-level structure to suppose $M$ quasiprojective. Then the previous proposition with $C$ the pull-back of the universal family gives $\PABL(M) = \emptyset$, and we have $\PBL(M) \subset \PABL(M)$.
\end{proof}

\begin{proof}[Proof (of Theorem \ref{thm:gen})]
	Let $M$ be a quasiprojective smooth complex algebraic variety with a finite étale covering map $M \to \Mg(\C)$. A family of smooth projective curves of genus $g \geq 2$ over $\Lsav$ induces by base change a map $h: \Lsav' \to M$ where $\Lsav'$ is a finite étale covering of $\Lsav$. Hence $\Lsav'$ verifies the same hypothesis as $\Lsav$. Theorem \ref{thm:main} and the definition of $\Charac$ gives that $\PBL(\Lsav')$ is linearly non-degenerate. Moreover the previous lemma gives that $\PBL(M)$ is empty. Hence Proposition \ref{prop:cste} applied to $h$ (in characteristic $0$, the separability hypothesis is always verified) gives the conclusion.
\end{proof}

This theorem has been proved in the compact case in \cite[Corollary 1.6]{liu2017curvatures}. Their proof relies on the use of Mok's metric rigidity theorem (\cite[Theorem 1]{mok1987uniqueness}). There actually does not seem to be any obstruction to apply the generalized version of To \cite[Theorem 1]{to1989hermitian} of the metric rigidity theorem to keep the same proof for the non compact case.

\bigskip

Hence, the usefulness of Theorem \ref{thm:main} and Proposition \ref{prop:cste} over $\C$ would be revealed in cases where we consider morphisms $\Lsav \to \Qpv$ whose target has an empty stable (or augmented) base locus, but does not have an obvious metric of Griffiths seminegative curvature, like in Theorem \ref{thm:bro}.

\begin{proof}[Proof (of Theorem \ref{thm:bro})]
	The fact that, in this case, $\PABL(\Prv \setminus \Div)$ is empty is exactly \cite[Theorem A]{brotbek2018positivity}. Then we apply Proposition \ref{prop:cste} with Theorem \ref{thm:main}.
\end{proof}

\subsection{The general case}

Now we can prove Theorem \ref{thm:mainapp}. For $g, h \geq 2$, recall that $\Ag$ (resp. $\Mh$) is the moduli stack of principally polarized Abelian varieties of dimension $g$ (resp. connected smooth projective curves of genus $h$). There exist a non empty irreducible scheme $S$ with an étale morphism $S \to \Spec(\Z)$ (resp. $T \to \Spec(\Z)$), a scheme $\bar{A}$ (resp. $\bar{M}$) smooth and proper over $S$ (resp. $T$), the complement of a normal crossing divisor $A \subset \bar{A}$ (resp.  $M \subset \bar{M}$) and a finite étale covering map $A \to \Ag \times S$ (resp. $M \to \Mh \times T$). For example, we can take $S = \Spec(\Z[1/3, e^{2\im \pi/3}])$, $T = \Spec(\Z[1/3])$, $A$ the moduli space of principally polarized Abelian varieties of dimension $g$ with $3$-level structure and $M$ the moduli space of connected smooth projective curves of genus $h$ with a convenient non-Abelian level structure, see \cite{pikaart1995moduli}. Hence they satisfy the hypothesis of Proposition \ref{prop:spe}.

\begin{proof}[Proof (of Theorem \ref{thm:mainapp})]
	Thanks to Lemma \ref{lem:vie}, $\PBL(M \otimes \C) = \emptyset$, hence $\PmBL(M \otimes \C) = \emptyset$ for some $m \in \Ns$. Thanks to Propositions \ref{prop:ext} and \ref{prop:spe}, for all but finitely many prime numbers $p$, for every field $k$ of characteristic $0$ or $p$, $\PmBL(M \otimes k) = \emptyset$.
	
	As the universal covering of the smooth complex algebraic variety $A \otimes_{\Z[e^{2\im \pi/3}]} \C$ is biholomorphic to $g$-th Siegel upper half-space, Theorem \ref{thm:main} implies that $\PBL(A \otimes_{\Z[e^{2\im \pi/3}]} \C)$ is linearly non-degenerate. Thus $\BL_n^{\Proj}(A \otimes_{\Z[e^{2\im \pi/3}]} \C)$ is linearly non-degenerate for every $n \leq md$, with $d$ the degree of the finite étale covering map $M \to \Mh$. Thanks to Propositions \ref{prop:ext} and \ref{prop:spe}, for all but finitely many prime numbers $p$, for every field $k$ of characteristic $0$ or $p$, for every connected component $B$ of $A \otimes k$, $\BL_n^{\Proj}(B)$ is linearly non-degenerate for every $n \leq md$.
	
	Consider a separable family of smooth projective curves over $B \otimes k$ associated with a morphism $B \otimes k \to \Mh \otimes k$. It lifts as a morphism of $k$-varieties $f: C \to M \otimes k$ with $C$ a finite étale covering of $B \otimes k$ of degree at most $d$. By Proposition \ref{prop:covsbl}, $\PmBL(C)$ is linearly non-degenerate. We conclude with Proposition \ref{prop:cste}.
\end{proof}

\printbibliography

\end{document}